\documentclass[11pt,a4paper,reqno]{amsart}
\usepackage{cmap}
\usepackage[T1]{fontenc}
\usepackage{graphicx}
\usepackage{subcaption}
\usepackage{setspace,mathrsfs,amsthm,amsmath,amssymb,amsfonts,lipsum,appendix,mathtools,cite,tabularx,fontenc,bbm,caption,soul}
\usepackage[normalem]{ulem}
\def\inpr#1{\left\langle #1\right\rangle}
\usepackage[shortlabels]{enumitem}
\makeatletter
\newtheorem*{rep@theorem}{\rep@title}
\newcommand{\newreptheorem}[2]{%
	\newenvironment{rep#1}[1]{%
		\def\rep@title{#2 \ref{##1}}%
		\begin{rep@theorem}}%
		{\end{rep@theorem}}}
\makeatother
\usepackage{amsthm}
\newtheorem*{theorem*}{Theorem}
\newtheorem{theorem}{Theorem}[section]
\newreptheorem{theorem}{Theorem}

\newtheorem{corollary}[theorem]{Corollary}
\newtheorem{proposition}[theorem]{Proposition}
\theoremstyle{definition}
\newtheorem{remark}[theorem]{Remark}
\newtheorem{definition}[theorem]{Definition}

\newcommand{\norm}[1]{\left\lVert#1\right\rVert}
\newcounter{A}
\renewcommand{\theA}{($\textbf{A}$)}

\newtheorem*{openproblem}{Open problem}
\newtheorem{example}[theorem]{Example}

\usepackage[nointegrals]{wasysym}
\usepackage[misc]{ifsym}
\usepackage[dvipsnames]{xcolor}
\definecolor{ao}{rgb}{0.0, 0.5, 0.0}
\definecolor{lasallegreen}{rgb}{0.0, 0.3, 0.0}
\usepackage{hyperref,tikz}
\hypersetup{
	colorlinks=true,
	linkcolor=blue,
	filecolor=magenta,      
	urlcolor=violet,
	citecolor=OrangeRed,
}
\usepackage[margin=0.75in]{geometry}
\usetikzlibrary{calc,patterns,angles,quotes}

\usepackage{wrapfig}

\newcommand{\cpom}{\mathcal{C}^+(\overline{\Omega})}
\let\oldnorm\norm
\def\norm{\@ifstar{\oldnorm}{\oldnorm*}}

\newcommand{\Om} {\Omega}
\newcommand{\la} {\lambda}

\newcommand{\Lom} {\mathcal{L}}

\newcommand{\ol}[1] {\overline{#1}}

\newcommand{\R}{{\mathbb R}}

\newcommand{\N}{{\mathbb N}}

\newcommand{\lb} {\langle}
\newcommand{\rb} {\rangle}

\usepackage[hyperpageref]{backref}

\usepackage{orcidlink}

\usepackage{a4wide}
\usepackage{csquotes}

\usepackage{amsfonts}
\usepackage{amssymb,amsthm}
\usepackage{amsmath}
\allowdisplaybreaks
\usepackage{mathtools}
\mathtoolsset{showonlyrefs}

\numberwithin{equation}{section}

\usepackage {setspace}

\usepackage{enumitem}

\setlist{nosep}

\usepackage{todonotes}

\begin{document}
\singlespacing

\title[On the weighted logarithmic potential operator]{On the weighted logarithmic potential operator}

\author[T.V.~Anoop]{T.V.~Anoop}
\author[Jiya Rose Johnson]{Jiya Rose Johnson}

\address[T.V.~Anoop]{\newline\indent
	Department of Mathematics,
	Indian Institute of Technology Madras, 
	\newline\indent
	Chennai 36, India
	\newline\indent
	\orcidlink{0000-0002-2470-9140} 0000-0002-2470-9140 
}
\email{anoop@iitm.ac.in}
\address[Jiya Rose Johnson]{\newline\indent
	Department of Mathematics,
	Indian Institute of Technology Madras, 
	\newline\indent
	Chennai 36, India
}
\email{jiyarosejohnson@gmail.com}


\begin{abstract} 
    For a bounded open set $\Omega \subset \mathbb{R}^N$ with $N\geq 2$, and for  positive continuous functions $w,g$ on $\overline{\Omega}$, we consider the  weighted eigenvalue problem
    \begin{equation}
        \mathcal{L}_{w} u =\tau gu,
    \end{equation}
    where $\mathcal{L}_w$ is the weighted logarithmic potential operator on $L^2(\Omega)$ as defined below:
    \begin{equation*}
        \mathcal{L}_w u(x)=\int_\Omega \log\left(\frac{w(x)w(y)}{|x-y|}\right)u(y)dy.
    \end{equation*}
    We study the monotonicity and continuity of the largest positive eigenvalue $\tau_{w,g}^+(\Omega)$ with respect to  $\Omega$, $w$, and $g$. We also establish that $\tau_{w,g}^+(\Omega)$ satisfies a reverse Faber–Krahn inequality under polarization. 
     We provide a sufficient condition for the existence of a negative eigenvalue in terms of the weighted transfinite diameter of $\Omega$, under the assumption that $\log(w)$ is superharmonic. For $\Omega\subset \mathbb{R}^2$, if $\Delta\log w $ is a constant $C$, we show that 0 can be an eigenvalue of $\mathcal{L}_w$ only when $C=\frac{2\pi}{|\Omega|}$. For such domains, if $\log(w)$ is a harmonic function on $\Omega$, we provide a representation formula for the eigenfunctions. Using this representation, we establish variants of the maximum principles that give some insight into the geometry of these eigenfunctions.
\end{abstract} 
\maketitle
\definecolor{lblack}{gray}{0.3}
\definecolor{mygray}{gray}{0.9}
\definecolor{vlgray}{gray}{0.96}
\definecolor{medgray}{gray}{0.8}
\definecolor{dgray}{gray}{0.7}
\begin{quote}	
	\setcounter{tocdepth}{1}
	\tableofcontents
	\addtocontents{toc}{\vspace*{0ex}}
\end{quote}
\section{Introduction}
\par For a bounded open set $\Om$ in $\R^2$, the logarithmic potential operator $\Lom$ on $L^2(\Omega)$ is defined as follows: 
\begin{equation}\label{definition_Lu}
    \Lom u(x)=\int_\Omega \log\frac{1}{|x-y|}u(y)dy,
\end{equation}
where $|\cdot|$ denotes the Euclidean norm in $\R^2$. It is easy to verify that the operator $\Lom$ is compact and self-adjoint (see \cite[p.367]{troutman1967}). Consequently, the spectrum of $\Lom$ comprises a countably infinite number of eigenvalues, with 0 being the only limit point, and the corresponding eigenfunctions form an orthonormal basis for $L^2(\Omega).$  The spectral properties of the logarithmic potential operator have been studied in numerous articles, for example, see  \cite{anoopjiya2025,troutman1967,troutman1969,suraganlog2016,Kac1970} and the references therein. In the planar case, the operator $\Lom$ is the Newtonian potential operator. In higher dimensions, $\Lom$ is an analogue of Riesz potential operator. The eigenvalue problems for Riesz potential operators have been studied in \cite{Alsenafi_etal_2024,suragan_ruzhansky_2016,suragan2016riesz,suragan2009}. A variant of the operator $\Lom$ is used to define the logarithmic capacity for compact sets in $\R^2$\cite{Ransford1995}. Here, we consider a two-way generalization of the logarithmic potential operator: 1) a weighted version of the logarithmic kernel and 2) higher-dimensional cases. 

Let us first fix some notation. We use $|\cdot|$ to denote the Euclidean norm in $\R^N$, which corresponds to the absolute value for $N=1$. We denote $B_r$ for the open ball centered at the origin with radius $r$, and $B_r(x)$ denotes the open ball centered at $x$ with radius $r$. For a compact set $E$ in $\R^N$, $\mathcal{C}^{+}(E)$ denotes the set of  all positive continuous function on $E$. Unless otherwise specified, $\Omega$ is a bounded open subset of $\R^N$ with $ N\ge 2$. 
  
For a weight function $w\in \cpom$, we define the weighted logarithmic potential operator $\Lom_w$ as follows:
\begin{equation}\label{def_lw}
    \Lom_w u(x):=\int_\Omega \log\left(\frac{w(x)w(y)}{|x-y|}\right)u(y)dy.
\end{equation}
As in the case of $\Lom$, we can verify that the operator $\Lom_w$ is compact and self-adjoint (see Proposition \ref{prop_Lom_compact}) and hence the set of all eigenvalues of $\Lom_w$ is a countable set with 0 being the only limit point. Moreover, the corresponding eigenfunctions form an orthonormal basis for $L^2(\Omega).$ In \cite{saff1992paper}, Saff considered a potential operator closely related to  $\Lom_w$,  defined as follows:
\begin{equation*}
    U_\mu(x):=\int_E \log\left(\frac{w(x)w(y)}{|x-y|}\right)d\mu(y),
\end{equation*}
where $\mu$ is a normalized Borel measure on the compact set  $E$. Using the associated weighted energy, Saff established the equivalence of the weighted capacity, the weighted transfinite diameter, and the weighted Chebyshev constant (see also \cite{ebsaff1997}). Similar weighted nonlocal operators, for instance, weighted Riesz potential–type operators, were investigated in \cite{Saff2007Riesz,Saff2009Riesz,Saff2023Riesz}. In this article, we study the following weighted eigenvalue problem:
\begin{equation}\label{evproblem}
    \Lom_w u =\tau\, gu,\quad\text{for }w,g\in\cpom.
\end{equation}
The eigenvalue problems with indefinite weights arise in a wide range of applications in engineering, physics, and biology; see, for example, \cite{sovrano2018} for migration-selection models in population genetics. The weighted eigenvalue problems for both the Laplacian and the $ p$-Laplacian have been extensively studied in the literature; for example, see \cite{Szulkin1999,Alberico2020} and the references therein. 
The weighted eigenvalue problem for the fractional Laplacian was studied in \cite{mrityunjoy2025}. We did not find any studies exploring the weighted-eigenvalue problem (or even eigenvalue problem) for $\Lom_w$.

We say $\tau$ is an eigenvalue of \eqref{evproblem} and $u\in L^2(\Om)\setminus\{0\}$  is an eigenfunction corresponding to $\tau$, if the following equation holds:
 \begin{equation}
    \lb \Lom_w u,v \rb=\tau \lb gu,v\rb,\quad \forall\, v\in L^2(\Om),
\end{equation}
where $\lb \cdot,\cdot\rb$ denotes the usual inner product on $L^2(\Om)$. In other words,
\begin{equation}\label{eulerlagrange}
    \iint\limits_{\Om\;\Om} \log\left( \frac{w(x)w(y)}{|x-y|}\right)u(x)v(y)dxdy=\tau\int_\Om guv,\quad \forall\, v\in L^2(\Om).
\end{equation}

It is easy to see that there is a one-to-one correspondence between the eigenvalues of \eqref{evproblem} and the eigenvalues of a compact operator $\frac{1}{g}\Lom_w$.  However, $\frac{1}{g}\Lom_w$ is not a self-adjoint operator, unless $g$ is a constant function (see Proposition \ref{prop_sa_iff}). Therefore, we will associate eigenvalues of \eqref{evproblem} with eigenvalues of another operator, which is self-adjoint and compact. Our first theorem asserts the existence of an infinite number of eigenvalues for \eqref{evproblem}. 
\begin{theorem}\label{thm_infinite_ev}
    Let $\Om\subset\R^N$ be a bounded domain and $w,g\in\cpom$. Then the set of non-zero eigenvalues of  \eqref{evproblem} forms a sequence $(\tau_n)$ and the corresponding eigenfunctions $(u_n)$ are in $\mathcal{C}(\overline{\Omega})$ such that  $(|\tau_n|)$ is decreasing and  has the following variational characterization:
    \begin{equation}
    \tau_n = 
    \max_{\substack{u \in  U_{n-1} \\ u \neq 0}}
    \frac{\langle \Lom_{w} u,\, u \rangle}{\int_{\Omega} gu^2 }, \;\;\text{ if } \tau_n>0,
    \end{equation}
    \begin{equation}
    \tau_n = 
    \min_{\substack{u \in U_{n-1} \\ u \neq 0}}
    \frac{\langle \Lom_{w} u,\, u \rangle}{\int_{\Omega} gu^2 }, \;\; \text{ if } \tau_n<0,
    \end{equation}
    where $U_0=L^2(\Omega)$ and $$U_n = \left\{u\in L^2(\Omega):\int_\Omega g u u_i=0, \forall\, i=1,2\ldots, n \right\}.$$
\end{theorem}

Next, we study some properties of the largest positive eigenvalue of \eqref{evproblem}. 

\subsection{The largest positive eigenvalue of \eqref{evproblem}} 
Let $\Om$ be a bounded open subset of $\R^N$, and let $w,g\in\cpom$.  We denote the largest positive eigenvalue of \eqref{evproblem}  by  $\tau_{w,g}^+(\Om)$. Therefore,
\begin{equation}\label{equivalent_definition}
  \tau_{w,g}^+(\Om)=\max_{\substack{u \in L^2(\Omega) \\ u \neq 0}}
 \frac{\lb \Lom_w u,u \rb}{\int_\Om gu^2}.
\end{equation}
Equivalently,
\begin{align}\label{weightedcharmu1}
       \tau_{w,g}^+(\Om)&=\max_{u\in \mathcal{M}_g} \lb \Lom_w u,u \rb,\quad \text{where}\quad \mathcal{M}_g=\left\{u\in L^2(\Omega):\int_\Om gu^2=1  \right\}.
\end{align}
Observe that the largest positive eigenvalue need not be the eigenvalue with the largest absolute value. In our next theorems, we explore the monotonicity and continuity of $\tau_{w,g}^+(\Om)$  with respect to the weights $w, g$ and the domain $\Om$. The domain monotonicity of the eigenvalues of $\Lom$ was established in~\cite[p.11]{Alsenafi_etal_2024}, while the strict domain monotonicity was proved in~\cite[Theorem 5.6]{anoopjiya2025}.  
The following theorem establishes the monotonicity of $\tau_{w,g}^+(\Omega)$ with respect to the weights $w, g$ and the domain $\Omega$.
\begin{theorem}\label{theorem_monotonicity_w_g}
Let $\Om$ and $\Om_1$ be a bounded open subsets of $\R^N$, and let $w,w_1,g,g_1\in\cpom$ be with $ \sqrt{\mathrm{diam}(\Om)}\leq w$.
    \begin{enumerate}[(i)]
        \item  Let $ w\leq w_1$ and  $g_1\leq g$. Then
                \begin{equation}
                    \tau_{w,g}^+(\Om)\leq \tau_{w_1,g_1}^+(\Om),
                \end{equation}
                and the equality holds only if $w=w_1$ and $g=g_1$.

        \item \label{item:dom:mon}
       Let $\Om_1\subseteq\Om$.  Then
                \begin{equation}
                    \tau_{w,g}^+(\Omega_1)\leq \tau_{w,g}^+(\Omega),
                \end{equation}
                and the equality holds only if $\Om=\Om_1$.
                   
    \end{enumerate}     
\end{theorem}

The following theorem asserts certain continuity properties of $\tau_{w,g}^+(\Omega)$ with respect to weights $w, g$ and domain $\Omega$. 
\begin{theorem}\label{theorem_convergence}
  For $n\in\N_0,$ let $\Om_n$ be open sets in $\R^N$ such that 
    $ \Om_n\subseteq B_R$, for some $R>0$.  Let
    $w_n,g_n\in\mathcal{C}^+(\ol{B_R})$, $\forall\, n\in\N_0$. Further,  assume that 
    \begin{enumerate}[(i)]
        \item  $\chi_{\Omega_n}\to \chi_{\Omega_0}$ pointwise on $\ol{B_R}$,
        \item   $(w_n,g_n)\to (w_0,g_0)$ uniformly on $\ol{B_R}$.
    \end{enumerate}
     Then, $$\tau_{w_n,g_n}^+(\Omega_n)\to \tau_{w_0,g_0}^+(\Om_0).$$ 
\end{theorem}

\subsubsection{Reverse Faber-Krahn inequality for $\tau_{w,g}^+$} 
 Recall that for a bounded domain $\Om\subset\R^N$, the first Dirichlet eigenvalue $\la_1(\Om)$ of the Laplacian satisfies the following Faber-Krahn inequality: 
    \begin{equation}\label{FK} 
        \la_1(\Omega^*)\leq \la_1(\Omega), 
    \end{equation} 
  where $\Omega^*$ is the open ball centered at the origin with the same measure as $\Omega$. For doubly connected planar domains, Payne and Weinberger \cite{Payne1961}, along with Hersch \cite{Hersch1963}, established the following reverse-type inequality for the first eigenvalue of the Laplacian under Dirichlet-Neumann boundary conditions:
  \begin{equation}\label{RFK} 
        \la_1(\Omega)\leq \la_1(\Omega^\#), 
    \end{equation} 
  where $\Omega^\#$ is an annular region that has the same measure as $\Omega$ and a fixed inner or outer perimeter depending on the Dirichlet boundary.  In the works \cite{AnoopAshok2020,Anoop_Ghosh_2024}, these findings have been extended to the first eigenvalue of the \(p\)-Laplacian under Dirichlet-Neumann boundary conditions in higher-dimensional domains containing holes. Similar results for the Laplacian with other boundary conditions (Neumann and Robin) have also been established in \cite{AnoopDrabekBobkov2025}.
  
 In \cite[Theorem 1.3]{Anoop-Ashok2023}, the authors proved a Faber-Krahn inequality under polarization, namely, 
\begin{equation*}
    \la_1(P_H(\Omega))\leq \la_1(\Omega),
\end{equation*}
where $P_H(\Omega)$ denotes the polarization of $\Omega$ (see Definition~\ref{def_pol}). In \cite[Theorem 1.6 $\&$ Theorem 1.8]{anoopjiya2025}, we  established the following reverse Faber–Krahn inequalities for the largest eigenvalue $\tau_{1,1}^+$ of $\Lom$, under polarization: 
\begin{equation}\label{logpolresult}
\tau_{1,1}^+(\Omega)\leq \tau_{1,1}^+(P_H(\Omega)), 
\end{equation} 
and also under Schwarz symmetrization: 
\begin{equation} 
\tau_{1,1}^+(\Omega) \leq \tau_{1,1}^+(\Omega^*). 
\end{equation}
We study analogous results for $\tau_{w,g}^+$. First, we recall the definition of polarization.

\begin{definition}\label{def_pol}{\textbf{(Polarization).}}
    A polarizer $H$ is an open affine halfspace in $\mathbb{R}^N.$ Let us denote the reflection with respect to the boundary $\partial H$ by $\sigma_H$. For a set $\Omega$ in $\R^N,$  the polarization $P_H(\Omega)$ with respect to $H$ is defined as follows:
    \begin{equation*}
        P_H(\Omega)=[(\Omega\cup \sigma_H(\Omega))\cap H]\cup [\Omega\cap\sigma_H(\Omega)]
    \end{equation*}
\end{definition}

 \begin{minipage}[t]{0.475\textwidth} 
      \begin{center}
        \captionsetup{type=figure}
            \begin{tikzpicture}
            \def\square1{(-2,-2) rectangle (2,2)};
                 \def\rect{(-3,0) rectangle (3,3)}
                    \fill[vlgray] \rect;
                    
                    \fill[dgray,rotate=-30] \square1;
                    \draw[black,rotate=-30] \square1;
        
                  \draw[-, dashed] (-3,0) -- (3,0);
                  \begin{scriptsize}   
                     \draw (-1.40,2) node {$\Om$};
                     \draw (2.75,2.75) node {$H$};
                 \end{scriptsize}
            \end{tikzpicture}
            \captionof{figure}{$\Om$}
        \end{center}
  \end{minipage}
  \begin{minipage}[t]{0.525\textwidth} 
    \begin{center}
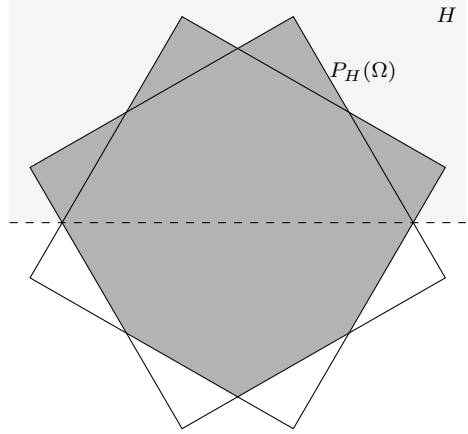

        \captionsetup{type=figure}
        \begin{tikzpicture}
        \def\square1{(-2,-2) rectangle (2,2)};        
             \def\rect{(-3,0) rectangle (3,3)}
                \fill[vlgray] \rect;
                \begin{scope}
                      \clip \rect;
                    \fill[dgray,rotate=30] \square1;
                    \fill[dgray,rotate=-30] \square1;
                \end{scope}
            
            \fill[dgray] (-2.31,0) -- (-1.46,-1.46) -- (0,-2.31) -- (1.46,-1.46) -- (2.31,0) -- cycle;
                \draw[black,rotate=-30] \square1;
                \draw[black,rotate=30]\square1;
    
              \draw[-, dashed] (-3,0) -- (3,0);
              \begin{scriptsize}   
                 \draw (1.65,2) node {$P_H(\Omega)$};
                 \draw (2.75,2.75) node {$H$};
             \end{scriptsize}
        \end{tikzpicture}
        \captionof{figure}{  $P_H(\Omega)$}
    \end{center}
 \end{minipage}

\begin{definition}{\textbf{(Polarization of a function).}}
    For a measurable function $u:\mathbb{R}^N\rightarrow \mathbb{R}$, the polarization $P_H(u)$ with respect to $H$ is defined as 
   \begin{equation*}
       P_H(u)(x)=\begin{cases}
           \max\{u(x),u(\sigma_H(x))\},& \text{   for }x\in H,\\
           \min\{u(x),u(\sigma_H(x))\},& \text{   for }x\in \mathbb{R}^N\setminus H.
       \end{cases}
   \end{equation*}
   Let $u:\Omega\rightarrow \mathbb{R}$ and let $\widetilde{u}$ be its zero extension to $\mathbb{R}^N$. Polarization $P_H(u)$ is defined as the restriction of $P_H(\widetilde{u})$ to $P_H(\Omega)$. 
\end{definition}
 Next, we state the reverse Faber–Krahn inequality for the largest eigenvalue $\tau_{w,g}^+$ under polarization. Note that $|S|$ denotes the Lebesgue measure of any measurable set $S$.
\begin{theorem}\label{weightedfk}
    Let $\Omega \subset \R^N$ be a bounded open set, and let $H$ be a polarizer. Let  $w,g\in \mathcal{C}^+(\ol{\Om\cup P_H(\Om)})$ be such that  $w(x)=w(\sigma_H(x))$ and $g(x)=g(\sigma_H(x))$ for all $x\in\Om\cup\sigma_H(\Om)$. If an eigenfunction corresponding to $\tau_{w,g}^+(\Om)$ is positive on $\Om\setminus \sigma_H(\Om)$, then
    \begin{equation}\label{pollog}
            \tau_{w,g}^+(\Omega)\leq \tau_{w,g}^+(P_H(\Omega)).
        \end{equation}
         The equality holds in $\eqref{pollog}$ only if $|P_H(\Omega)\triangle\Omega|=0$ or $|P_H(\Omega)\triangle\sigma_H(\Omega)|=0$, where $\triangle$ is the symmetric difference of the sets.
\end{theorem}

\begin{remark}
    If  $ \sqrt{\mathrm{diam}(\Omega)}\leq w$ in $\Om$, in Proposition \ref{prop-non-neg-w} we establish that  $\tau_{w,g}^+(\Om)$ admits a positive  eigenfunction on entire $\Om$. So, the conclusion of the above theorem holds if we replace the positivity assumption on the eigenfunction with the assumption $ \sqrt{\mathrm{diam}(\Omega)}\leq w$.  
\end{remark}  
\begin{remark}
    Let $\Omega \subset \R^N$ be a bounded open set. If $w = g = c$, and  $\sqrt{\mathrm{diam}(\Omega)}\leq c$, for some constant $c$, then by applying \eqref{pollog} repeatedly, one can obtain the reverse Faber-Krahn inequality for $\tau_{w,g}^+$ under Schwarz symmetrization, i.e.
    \begin{equation} 
        \tau_{w,g}^+(\Omega)\leq \tau_{w,g}^+(\Omega^*) .
    \end{equation}
    The proof of this inequality follows the same lines as that of \cite[Theorem 1.8]{anoopjiya2025}, which deals with the case $N=2$ and $w=g=1$.
\end{remark}

Next, we discuss the existence of a negative eigenvalue for \eqref{evproblem}.
\subsection{The smallest negative eigenvalue of \eqref{evproblem}} 
In contrast to $\tau_{w,g}^+(\Om)$, now we define
\begin{equation}\label{equivalent_definition_neg}
     \tau_{w,g}^-(\Om)=\inf_{\substack{u \in L^2(\Omega) \\ u \neq 0}}
 \frac{\lb \Lom_w u,u \rb}{\int_\Om gu^2}.
\end{equation}
Equivalently,
\begin{align}\label{weightedcharneg}
      \tau_{w,g}^-(\Om)&=\inf_{u\in \mathcal{M}_g} \lb \Lom_w u,u \rb,\quad \text{where}\quad \mathcal{M}_g=\left\{u\in L^2(\Omega):\int_\Om gu^2=1  \right\}.
\end{align}
 If $\Omega$ is such that $\Lom_w$ is a positive operator on $L^2(\Omega)$, then it follows that $\tau_{w,g}^-(\Omega) = 0$ (see Proposition \ref{prop_tau_tilde_zero}). On the other hand, if $\Lom_w$ is not a positive operator on $L^2(\Omega)$, then $\tau_{w,g}^-(\Omega)$ is negative, and in fact it is the smallest negative eigenvalue of \eqref{evproblem}, by Theorem \ref{thm_infinite_ev}. 

In \cite{troutman1967}, for $\Om\subset\R^2$, Troutman related the existence of a negative eigenvalue for the operator $\Lom$ in terms of a geometric quantity known as the transfinite diameter of $\Om$. More precisely, in \cite[Theorem 3]{troutman1967}, the author stated that $\Lom$ has a negative eigenvalue if and only if $T_{diam}(\Om)$ (the transfinite diameter of $\Om$) is greater than 1. However, his proof has a flaw that cannot be rectified immediately. In 1970, Kac \cite{Kac1970} provided some probabilistic arguments to prove Troutman's theorem. In \cite[Theorem 1.12]{anoopjiya2025}, we provided an alternative analytic proof for this theorem. Here, we establish an analogous result for $\Lom_w$ in terms of the weighted transfinite diameter (for definition, see Subsection~\ref{subsec_weighted_trans}) of an open set $\Omega$ in $\R^N$ for any $N\ge 2$. The following theorem provides a sufficient condition for the existence of a negative eigenvalue.
 
\begin{theorem}\label{theorem_weighted_trans}
    Let $\Omega\subset\R^N$ be a bounded open set, and let $w \in \mathcal{C}^+(\overline{\Omega})$ be such that $\log w$ is superharmonic on $\Om$. If $T_{diam}\left(\ol{\Om},\frac{1}{w}\right)> 1$, then $\Lom_w$ admits a negative eigenvalue.
\end{theorem}
Indeed, the above theorem extends Theorem 3 of Troutman \cite{troutman1967} as well as the corresponding result of Kac \cite[p. 595]{Kac1970}. 
 \begin{openproblem}For a polarizer $H$, the  reverse Faber–Krahn type inequality for $\tau_{w,g}^-$
     \begin{equation}
         \tau_{w,g}^-(\Om) \leq \tau_{w,g}^-(P_H(\Om)), 
     \end{equation}
holds or not is an open problem, even for $w=g=1.$
 \end{openproblem}
 
\par Next, we study the convergence of negative eigenvalues $\tau_{w_n,g_n}^-(\Omega_n)$ under some suitable convergence of $\Om_n,w_n$ and $g_n$. For $N=2$ and  $w_n=g_n=1$, in \cite[Theorem 2]{troutman1969}, the author studied the convergence with respect to  $\Omega_n$. In this direction, we have the following theorem.

\begin{theorem}\label{theorem_convergence_negative}
    For $n\in\N_0,$ let $\Om_n$ be open sets in $\R^N$ such that 
    $ \Om_n\subseteq B_R$, for some $R>0$.  Let
    $w_n,g_n\in\mathcal{C}^+(\ol{B_R})$,  and let $\log w_n$ be superharmonic on $\ol{B_R}$. Further,  assume that 
    \begin{enumerate}[(i)]        
        \item  $T_{diam}\left(\Omega_n, \frac{1}{w_n}\right) > 1,\quad \forall\, n\in\N_0.$
        \item  $\chi_{\Omega_n}\to \chi_{\Omega_0}$ pointwise on $\ol{B_R}$.
        \item  $(w_n,g_n)\to (w_0,g_0)$ uniformly on $\ol{B_R}$.        
    \end{enumerate}
     Then, $$\tau_{w_n,g_n}^-(\Omega_n)\to \tau_{w_0,g_0}^-(\Omega_0).$$ 
\end{theorem}

\subsection{On the planar case}
This subsection presents results that are valid only in the planar setting, as they rely on the fact that $\Delta \Lom u = -2\pi u$, see \cite[Lemma 4.2]{gilbarg_trudinger}, and the mean value property of the harmonic function $\log \frac{1}{|x|}$. We start by answering whether 0 is an eigenvalue of \eqref{evproblem}. In the planar case where \( w = g = 1 \), Kac observed that since \( \Delta \Lom u =-2\pi  u \), the value 0 is not an eigenvalue (see \cite[p.595]{Kac1970}). We use the same property to provide a more general sufficient condition on \( w \) to ensure that 0 is not an eigenvalue of equation \eqref{evproblem}.

\begin{theorem}\label{theorem_0_ev}
    Let $\Omega \subset \mathbb{R}^2$ be a bounded open set, and let $w,g\in\cpom$ be such that $\Delta (\log w )= C$, a constant. If $C\neq \frac{2\pi}{|\Om|}$, then 0 is not an eigenvalue of \eqref{evproblem}. 
\end{theorem}
\begin{remark}
   It can be seen from the above theorem that 0 is not an eigenvalue of \eqref{evproblem} if $\log w$ is a harmonic function. In particular, Troutman’s result \cite[Corollary~1]{troutman1967} is a consequence of this theorem. 
\end{remark}
        
\begin{remark}\label{remark_w1_w2}
    For the case $\Delta \log w =\frac{2\pi}{|\Om|}$, we cannot draw any conclusions in general. For example, on a disk $B_R$, consider the weight function\footnote{An earlier version of this example was generated with the assistance of ChatGPT (OpenAI); all mathematical verification and final formulation were carried out by the authors.} $$w_1(x)=e^{\frac{|x|^2}{2R^2}+\frac{1}{2}\log R-\frac{3}{8}}.$$  We show that (see Example \ref{example_disk}), for any $c>0,$  $\Delta (\log  cw_1 )= \frac{2\pi}{|B_R|}$ and 0 is not an eigenvalue of \eqref{evproblem} for $w=cw_1,$ wih $\; c\neq 1$. However, 0 is an eigenvalue of \eqref{evproblem} for $w=w_1$. 
\end{remark}

 In \cite[p.~367]{troutman1967}, the eigenfunctions of \eqref{evproblem} for $w = g = 1$ are represented by complex analytic methods. Using the same approach, we derive the following representation for the eigenfunctions of \eqref{evproblem}.

\begin{theorem}\label{efreptheorem}
Let $w,g\in\cpom$, and let $\log w$ be a harmonic function on $\Om$. Let $\tau$ be an eigenvalue of \eqref{evproblem}, and $\phi$ be an associated eigenfunction. Then, for 
every $x \in \Om$ and for each $r > 0$ such that $B_r(x)\subset \Om$, the following representation holds:
            \begin{equation}\label{eqn:repn}
            \tau \, \phi(x) g(x) = \frac{\tau}{2\pi} \int_{\partial B_1} \phi(x + r \zeta) g(x + r \zeta) \, dS(\zeta) - \int_{ B_r(x)} \log\left( \frac{|x - y|}{r} \right) \phi(y) \, dy.
            \end{equation}
\end{theorem}
From this representation, we derive a maximum principle for the eigenfunctions corresponding to negative eigenvalues. This provides geometrical insights about the eigenfunctions, as stated in Corollary~\ref{cor_efrep} and Remark \ref{rem:efrep}.

 This article is organized as follows. In Section \ref{sec_prelims}, we include the preliminary results. In Section \ref{sec_mains}, we prove the theorems stated in the introduction that hold for the general dimension $N$: Theorem \ref{thm_infinite_ev}, Theorem \ref{theorem_monotonicity_w_g}, Theorem \ref{theorem_convergence}, Theorem \ref{weightedfk}, Theorem \ref{theorem_weighted_trans}, and Theorem \ref{theorem_convergence_negative}. In Section \ref{sec_twodim}, we establish additional results for the planar case: Theorem \ref{theorem_0_ev} and Theorem \ref{efreptheorem}.
\section{Preliminaries}\label{sec_prelims}
In this section, we prove results relevant to our discussions or required for the proofs of our main theorems. Throughout this paper, for the eigenvalue problem \eqref{evproblem}, we consider weight functions $w,g \in \cpom$.  For convenience, we use $m$ and $M$ to denote the minimum and maximum of these weight functions, respectively. We start with a proposition that ensures that for any open set $V\subset \Om$, the operator $\Lom_w|_V$ is nonzero.
\begin{proposition}\label{prop:nonzero}
Let $w\in \cpom$ and for $x_1,x_2\in \Omega$ such that $x_1\neq x_2.$ Then the set $$A(x_1,x_2):=\left\{y\in \Omega\setminus\{x_1,x_2\} :\frac{w(x_1)w(y)}{|x_1-y|}=\frac{w(x_2)w(y)}{|x_2-y|}\right\}$$ does not have an interior point. In particular, for any open subset $V$ of $\Omega,$ there can exist at most one $x$ such that $\frac{w(x)w(y)}{|x-y|}=1, \forall\,y\in V.$ 
\end{proposition}
\begin{proof}
    Clearly,
    \begin{equation}
       A(x_1,x_2)=\{y\in \Omega\setminus\{x_1,x_2\} :\ |x_1-y|=C|x_2-y|\},
    \end{equation}
    where $C=\dfrac{w(x_1)}{w(x_2)}$. By completing the square, we see that the set $A(x_1,x_2)$ lies in either a hyperplane (if $C=1$) or a sphere (if $C\neq1$)\cite[Theorem 3.2.14, p.45]{hitchman}; in either case, $A(x_1,x_2)$ has an empty interior. 
\end{proof}

\begin{proposition}\label{prop_sa_iff}
    Let $g\in \cpom$. Then, the operator $\frac{1}{g}\Lom_w$ is self-adjoint on $L^2(\Omega)$ if and only if $g$ is constant.
\end{proposition}

\begin{proof}
     If $g$ is constant, then clearly $\frac{1}{g}\Lom_w$ is self-adjoint. Conversely, assume that $\frac{1}{g}\Lom_w$ is self-adjoint. For $u\in L^2(\Om)$,  let us define $G(u)=\frac{1}{g}u.$ Since both $G$ and $\Lom_w$ are self-adjoint operators, $\frac{1}{g}\Lom_w$ is self-adjoint if and only if $G$ and $\Lom_w$ commute. Thus,  $\Lom_w\circ G=G\circ \Lom_w$. For $u\in L^2(\Om)$, we have 
    \begin{align}
        (\Lom_w\circ G)(u)(x)&=\Lom_w\left(\frac{u}{g}\right)(x)=\int_\Om \log\left(\frac{w(x)w(y)}{|x-y|}\right)\frac{u(y)}{g(y)}dy,
        \\ (G\circ \Lom_w)(u)(x)&= \frac{1}{g(x)}\Lom_wu(x)=\int_\Om \log\left(\frac{w(x)w(y)}{|x-y|}\right)\frac{u(y)}{g(x)}dy.
    \end{align}
If possible, let $x_1,y_1\in\Om$ be such that $g(x_1)\neq g(y_1)$. Then, we can find $r>0$ such that $g(x)\neq g(y)$ for every $x\in B(x_1, r)$ and $y\in B(y_1,r)$. By Proposition \ref{prop:nonzero}, there exists at most one point $x\in \Omega\setminus B(y_1,r)$ such that $\frac{w(x)w(y)}{|x - y|}\equiv 1.$ Thus there exist $x_0\in  B(x_1, r)$ and $y_0\in B(y_1,r)$ such that $\frac{w(x_0)w(y_0)}{|x_0 - y_0|}\neq 1$. 
Define  $$u_0(y)=\left(\frac{1}{g(y)}-\frac{1}{g(x_0)}\right)\log\left(\frac{w(x_0)w(y)}{|x_0 - y|}\right),\quad y\in B(y_1,r).$$ 
Thus $u_0(y_0)\neq 0,$ and by continuity $u_0\neq 0$ in a neighborhood of $y_0$. Hence
\begin{equation*}
        (\Lom_w\circ G)(u_0)(x_0)-(G\circ \Lom_w)(u_0)(x_0)=\int_{\Omega} u_0^2dy\ge \int_{B(y_1,r)} u_0^2dy> 0.
    \end{equation*} 
 This is in contradiction to the commutativity of the operators. Thus, $g$ must be constant on $\Omega.$
\end{proof}
  
The following proposition is standard; a proof can be found in many functional analysis books. For example, see Proposition 8.5.2 and Theorem 8.5.2 of \cite{Keshfunctional2022}. We will omit the proof here.
\begin{proposition}\label{prop_kesh_func}
Let $A$ be a compact self-adjoint operator on a Hilbert space $H$ such that Range(A) is infinite dimensional. Then the nonzero eigenvalues of $A$ can be enumerated as $\la_n$ (repeated according to their multiplicities)  so that
\begin{equation}
    |\la_1|\geq |\la_2|\geq \cdots |\la_n|\geq \cdots.
\end{equation}
Moreover, there exists an orthogonal set 
$\{v_n: v_n \text{ is an eigenfunction corresponding to } \la_n, n\in \N\}$, and each $\la_n$ admits the following variational characterizations:
\begin{equation}
    \la_n= \frac{\langle A v_n,\, v_n \rangle}{\|v_n\|^2 }=
    \max_{\substack{v \perp  V_{n-1} \\ v \neq 0}}
    \frac{\langle A v,\, v \rangle}{\|v\|^2 },\quad \text{if } \la_n>0, 
    \end{equation}
    \begin{equation}
    \la_n=  \frac{\langle A v_n,\, v_n \rangle}{\|v_n\|^2 }=
    \min_{\substack{v \perp V_{n-1} \\ v \neq 0}}
    \frac{\langle A v,\, v \rangle}{\|v\|^2 }, \quad \text{if } \la_n<0, 
    \end{equation}
    where $V_0=\{0\}, V_n = {\rm Span}\{v_1,v_2,\ldots v_n\}$.
    Moreover, any extremizer of the above variational problems is an eigenfunction corresponding to $\la_n.$
\end{proposition}
Next, we recall the definition of the weighted transfinite diameter from \cite{ebsaff1997}.
\subsection{The weighted transfinite diameter}\label{subsec_weighted_trans}
Let $E$ be a compact set in $\R^N$ and $W\in \mathcal{C}^+(E)$ be a weight function. For each $n\geq 2$, let 
\begin{equation}\label{rho_n}
    \rho_n(E,W)= \max_{(x_1, x_2, \cdots x_n) \in E^n} \,\,  \left [\underset{1\leq i<j\leq n}{\prod} |x_i-x_j|W(x_i)W(x_j) \right]^{\frac{2}{n(n-1)}},
\end{equation}   
be the $n^{th}$-weighted diameter of $E$. Since $W$ is continuous on the compact set $E$, the maximum in \eqref{rho_n} is attained for some points $x_1, x_2, \dots, x_n \in E$. These points are called as the weighted Fekete points of $E$ associated with $W$. The sequence $\rho_n(E,W)$ is monotonically decreasing (see \cite[Theorem 1.1(p.143)]{ebsaff1997} for $n=2$). \emph{The weighted transfinite diameter of $E$}, denoted by $T_{diam}(E,W)$, is defined as 
\begin{equation}\label{limit_n_diameter}
    T_{diam}(E,W)=\lim_{n\rightarrow\infty}\rho_n(E,W).
\end{equation}  
For a non-empty bounded open set $\Omega\subset\R^N$, we define the weighted transfinite diameter as 
\begin{equation}
    T_{diam}(\Om,W):=T_{diam} (\overline{\Om},W).
\end{equation}

\begin{remark}\label{remark_weighted_trans}
We list some remarks related to the weighted transfinite diameter.
    \begin{enumerate}[(i)]
         \item For $W = 1$ and $N=2$, $T_{diam}(E,W)=T_{diam}(E)$, the classical transfinite diameter of $E$.

         \item For $c>0$,  $T_{diam}(E,cW)=c^2\,T_{diam}(E,W)$.

         \item For disk $B_R\subset\R^2 $,  $T_{diam}(B_R)=R$, while for weight  $W(x)=e^{-\left(\frac{|x|^2}{2R^2}+\frac{1}{2}\log R-\frac{3}{8}\right)}$, we have $T_{diam}(B_R,W)=1$ (see Proposition \ref{prop_ball_computation}). 
    \end{enumerate}
\end{remark}
\begin{remark}\label{alternate_definition}
    An alternative definition of the weighted transfinite diameter is given in terms of weighted logarithmic capacity(see \cite[p.27, p.63 $\&$ Theorem 1.3(p.145)]
        {ebsaff1997}). For a Borel measure $\nu$ on $E$, define
        \begin{equation}
            I_E(\nu,W):=\iint\limits_{E\;E} \log\left(\frac{1}{|x-y|W(x)W(y)}\right)d\nu(x)d\nu(y).
        \end{equation}
        The weighted Robin constant is given by
        \begin{equation}\label{robinconst}
           V(E,W)=\underset{\nu}{\inf} \, I_E(\nu,W),
        \end{equation}
        where $\nu$ varies over all normalized Borel measures on $E$. The infimum $V(E,W)$ is attained by a measure $\nu$ \cite[Theorem~3.1(b)]{saff1992paper}, known as the equilibrium measure. The weighted transfinite diameter is defined by  
        \begin{equation}\label{weightedlogarithmiccapacity}
            T_{diam}(E,W)=e^{- V(E,W)},
        \end{equation}
       which is also referred to as the weighted logarithmic capacity.
\end{remark}
For $N=2$, the convergence of the transfinite diameter for a monotone sequence of compact sets is proved in \cite[Theorem 5.1.3, p.~128]{Ransford1995}. In \cite[Theorem 3.3(d)]{saff1992paper}, this result is extended for the weighted transfinite diameter for the decreasing sequences of compact sets. We adapt the arguments from \cite{Ransford1995} and establish the monotone convergence for the weighted transfinite diameter for an increasing sequence of compact sets.
\begin{proposition}\label{thm:weighted_diam_limit}
    Let $E$ and $E_n$ (for $n\in \N$) be compact sets in $\R^N$ such that $  E_n \subseteq E_{n+1},\; \forall\,n\in\N$ and $E = \bigcup_{n=1}^\infty E_n.$  Let $W \in \mathcal{C}^+(E)$ be a weight function. Then 
    \begin{equation}
       \lim_{n \to \infty} T_{\mathrm{diam}}(E_n, W)= T_{\mathrm{diam}}(E,W).
    \end{equation}
\end{proposition}
\begin{proof}
    By the domain monotonicity of the transfinite diameter,
    \begin{equation}
        T_{\mathrm{diam}}(E_n,W) \leq T_{\mathrm{diam}}(E_{n+1},W)  \leq T_{\mathrm{diam}}(E,W),\quad\forall\, n\in\N,
    \end{equation}
    and hence
    \begin{equation}\label{trans_domain_monotonicity}
        \lim_{n \to \infty} T_{\mathrm{diam}}(E_n, W)\leq T_{\mathrm{diam}}(E,W).
    \end{equation}
    For the opposite inequality, we use the alternate definition of transfinite diameter(see Remark \ref{alternate_definition}). Let $\nu$ be the equilibrium measure of $E$ such that $$T_{\mathrm{diam}}(E,W)=e^{-I_E(\nu,W)}.$$
    Define $\mu_n = \frac{\nu|_{E_n}}{\nu(E_n)}$,
    which is a normalized Borel measure supported on $E_n$ and also on $E$. By the definition of the weighted transfinite diameter  \eqref{weightedlogarithmiccapacity}, it follows that
    \begin{equation}\label{eqn_trans_potential}
        e^{-I_E(\mu_n,W)}\leq T_{diam}(E_n,W).
    \end{equation}  
    Since $\nu(E_n) \to \nu(E)= 1$,  we use the dominated convergence theorem to obtain
    \begin{align}
        I_E(\mu_n,W) &=\frac{1}{\nu(E_n)^2}\iint\limits_{E\;E} \log\left(\frac{1}{|x-y|W(x)W(y)}\right)\chi_{E_n}(x)\chi_{E_n}(y) d\nu(x)d\nu(y)\to I_E(\nu,W),
    \end{align}    
    Now, from \eqref{eqn_trans_potential}, we obtain
    \begin{equation}
        T_{\mathrm{diam}}(E,W) = e^{-I_E(\nu,W)}= \lim_{n \to \infty}
           e^{-I_E(\mu_n,W)}\leq \lim_{n \to \infty}T_{\mathrm{diam}}(E_n,W).
    \end{equation}
    This completes the proof.   
\end{proof}

\section{Proof of the main theorems}\label{sec_mains}
In this section, we prove the results that hold for any dimension $N\geq 2$. We prove some propositions required for proving the main theorems. We then prove Theorem \ref{thm_infinite_ev}, Theorem~\ref{theorem_monotonicity_w_g}, Theorem \ref{theorem_convergence}, Theorem~\ref{weightedfk}, Theorem \ref{theorem_weighted_trans} and Theorem \ref{theorem_convergence_negative}. To prove Theorem \ref{thm_infinite_ev}, we consider a more general operator $\widetilde{\Lom}_{w,h}$ on $ L^{2}(\Omega)$ defined by
\begin{equation*} \widetilde{\Lom}_{w,h} u(x)=\int_\Om h(x,y)\log\left(\frac{w(x)w(y)}{|x-y|}\right)u(y)dy, \quad u\in L^2(\Om),\end{equation*}
where $w \in C^{+}(\overline{\Omega})$ and 
$h \in C^{+}(\overline{\Omega}\times\overline{\Omega})$ is symmetric, i.e., 
$h(x,y) = h(y,x)$.
In the next proposition, we establish some properties of this operator by adapting Troutman’s argument from \cite[p. 366]{troutman1967}.
\begin{proposition}\label{prop_Lom_compact}
    Let $\Om$ be a bounded domain in $\R^N$ and $w\in\cpom$. Then
    \begin{enumerate}[(i)]
        \item $\widetilde{\Lom}_{w,h} : L^2(\Omega) \to   L^2(\Omega)$ is a bounded linear operator.
        \item\label{prop_2} $\widetilde{\Lom}_{w,h} \, u\in C(\overline{\Om})$.
        \item\label{prop_3} $\widetilde{\Lom}_{w,h} $ is compact and self-adjoint.
    \end{enumerate}
\end{proposition}
\begin{proof}
$(i)$ For $u\in L^2(\Om)$, we first show that $\widetilde{\Lom}_{w,h} u\in L^2(\Om)$:
\begin{align}
    |\widetilde{\Lom}_{w,h} u(x)|^2&=\left| \int_\Om h(x,y)\log \left(\frac{w(x)w(y)}{|x-y|}\right)u(y)dy\right|^2
    \leq\|h\|_\infty^2\|u\|^2  \int_\Om \left| \log \left(\frac{w(x)w(y)}{|x-y|}\right)\right|^2 dy\nonumber\\
    &\leq 2\|h\|_\infty^2 \|u\|^2 \left[\int_\Om |\log (w(x)w(y))|^2dy+\int_\Om \left| \log \frac{1}{|x-y|}\right|^2dy\right].
\end{align}
 Since $w\in\cpom$, the function $\log(w(x)w(y))$ is bounded on $\Omega$, so the first integral is finite. For  $R>\mathrm{diam}(\Om)$, we have $\Om\subseteq B_R(x),\,\,\forall\, x\in \Om$. Thus,
\begin{equation*}
    \int_\Om \left| \log \frac{1}{|x-y|}\right|^2 dy\leq \int_{B_R(x)} \left| \log \frac{1}{|x-y|}\right|^2 dy = |\mathbb{S}^{N-1}| \int_0^R r^{N-1}(\log \, r)^2 dr<+\infty.
\end{equation*}
Therefore,
\begin{equation}\label{uniform_bound}
    |\widetilde{\Lom}_{w,h} u(x)|^2\leq K_{w,h,\Om}  \|u\|^2,
\end{equation}
for some constant $K_{w,h,\Om}$ that depends only on $w,h$ and $\Om$. Since $\Om$ is bounded, $\widetilde{\Lom}_{w,h} u\in L^2(\Om)$, and the operator $\widetilde{\Lom}_{w,h}$ is bounded on $L^2(\Om)$. 
\hfill\\
\\
$(ii)$ Let $x_1,x_2\in\Om$. Then
\begin{align}\label{ineq1}
    |\widetilde{\Lom}_{w,h} &u(x_1)-\widetilde{\Lom}_{w,h} u(x_2)|^2  
    \leq 2 \|u\|^2 \|h\|_\infty^2\int_\Om \left| \log \left(\frac{|x_2-y|}{|x_1-y|}\times\frac{w(x_1)}{w(x_2)}\right)\right|^2 dy\nonumber \\
    & +2 \|u\|^2\left(\int_\Om |h(x_1,y)-h(x_2,y)|^4dy\right)^{\frac{1}{2}} \left(\int_\Om \left| \log \left(\frac{w(x_2)w(y)}{|x_2-y|}\right)\right|^4 dy\right)^{\frac{1}{2}}.
\end{align}
To estimate the first term, define
\begin{equation}
    F(z_1,z_2,y):= \frac{|z_2-y|}{|z_1-y|}\times\frac{w(z_1)}{w(z_2)},\text{ for } z_1,z_2,y\in\Om,
\end{equation}
and 
\begin{align}
    \Om_1=\{y\in\Om :F(x_1,x_2,y)\geq 1 \} \text{ and }\Om_2=\{y\in\Om :F(x_2,x_1,y)> 1 \}.
\end{align}
Observe that
\begin{enumerate}[(a)]
    \item $F(z_1,z_2,y)=\frac{1}{F(z_2,z_1,y)}$ and  $y\in\Om_1$ or $y\in\Om_2$,
    \item $|\log(F(x_1,x_2,y))|=|\log(F(x_2,x_1,y))|,\quad \forall\, y\in\Om$.
\end{enumerate} 
To estimate the integrand in \eqref{ineq1}, we use the following inequality: 
\begin{equation}\label{logestimate}
    \log t \leq C_\alpha(t-1)^\alpha, \quad \forall\, t \geq 1, \text{ and } \alpha \in \left(0,1\right),
\end{equation}
which follows from the fact that the function $\frac{\log t}{(t-1)^\alpha}$ vanishes at 1 and $\infty$. For $y\in \Om_1$, we have
    \begin{align}
        F(x_1,x_2,y)- 1&=\frac{|x_2-y|w(x_1)-|x_1-y|w(x_2)}{|x_1-y|w(x_2)}\\
            &=\frac{|x_2-y|(w(x_1)-w(x_2))+w(x_2)( |x_2-y|-|x_1-y|)}{|x_1-y|w(x_2)}\\
            &\leq \frac{\mathrm{diam}(\Om)|w(x_1)-w(x_2)|+M |x_1-x_2|}{m|x_1-y|}\leq C\frac{|w(x_1)-w(x_2)|+ |x_1-x_2|}{|x_1-y|}.
    \end{align}
Now using  \eqref{logestimate}, and  the inequality $(|a| + |b|)^\alpha \leq 2^\alpha(|a|^\alpha + |b|^\alpha)$, we get
\begin{align*}
     \left| \log \left(\frac{|x_2-y|}{|x_1-y|}\times\frac{w(x_1)}{w(x_2)}\right)\right|&\leq 
       C_\alpha' \frac{(|w(x_1)-w(x_2)|^\alpha + |x_1-x_2|^\alpha)}{|x_1-y|^\alpha}.
\end{align*}
Similarly, for $y\in\Om_2$
\begin{align*}
     \left| \log \left(\frac{|x_2-y|}{|x_1-y|}\times\frac{w(x_1)}{w(x_2)}\right)\right|
     &\leq C_\alpha' \frac{(|w(x_1)-w(x_2)|^\alpha + |x_1-x_2|^\alpha)}{|x_2-y|^\alpha}.
\end{align*}
 Upon integrating over $\Om_i$, for $i=1,2$, we obtain
\begin{align}\label{p1}
    \int_{\Om_i} \left| \log \left(\frac{|x_2-y|}{|x_1-y|}\times\frac{w(x_1)}{w(x_2)}\right)\right|^2 dy&\leq  (C_\alpha')^2(|w(x_1)-w(x_2)|^\alpha + |x_1-x_2|^\alpha)^2 \int_{B(x_i,R)} \frac{1}{|x_i-y|^{2\alpha}}dy\nonumber\\
    & \leq  (C_{\alpha}')^2|\mathbb{S}^{N-1}| (|w(x_1)-w(x_2)|^\alpha + |x_1-x_2|^\alpha)^2 \int_{0}^R \frac{1}{r^{2\alpha-N+1}}dr.
\end{align}
The second term in \eqref{ineq1} can be estimated by observing that
\begin{align}\label{p2}
    \int_\Om \left| \log \left(\frac{w(x_2)w(y)}{|x_2-y|}\right)\right|^4 dy&\leq 2^4\left[ \int_\Om |\log (w(x_2)w(y))|^4dy + \int_{B_R(x_2)} \left| \log \left(\frac{1}{|x_2-y|}\right)\right|^4 dy\right]\nonumber\\
    &\leq 2^4\left[ C_\Om + |\mathbb{S}^{N-1}| \int_0^R r^{N-1}(\log \, r)^4 dr\right]=C_\Om ',
\end{align}
where $C_\Om $ and $C_\Om '$ are constants depending on $\Om$.  Substituting \eqref{p1} and \eqref{p2} in \eqref{ineq1}, we conclude that 
\begin{align}\label{uniformcontinuity}
    |\widetilde{\Lom}_{w,h} u(x_1)-\widetilde{\Lom}_{w,h} u(x_2)|^2&\leq C_\alpha'' \|u\|^2 (|w(x_1)-w(x_2)|^\alpha +  |x_1-x_2|^\alpha)^2\nonumber\\
    &\quad\quad+2 \sqrt{C_\Om '} \|u\|^2\left(\int_\Om |h(x_1,y)-h(x_2,y)|^4dy\right)^{\frac{1}{2}}.
\end{align}
Since $w$ is uniformly continuous on $\Om$ and $h$ is uniformly continuous on $\Om\times\Om$, it follows that $\widetilde{\Lom}_{w,h} \, u\in C(\overline{\Om})$.
\hfill\\
\\
\noindent $(iii)$ Let $(u_n)$ be a  bounded sequence in $ L^2(\Om)$. Then from \eqref{uniform_bound} and \eqref{uniformcontinuity}, we conclude that  $(\widetilde{\Lom}_{w,h} (u_n))$ is uniformly bounded and equicontinuous. By the Arzela–Ascoli theorem, $(\widetilde{\Lom}_{w,h} (u_n))$ has a  subsequence that converges in $C(\overline{\Om})$, and hence also converges in  $L^2(\Om)$. Therefore, the operator $\widetilde{\Lom}_{w,h} $ is  compact  on $L^2(\Om)$. Clearly,
\begin{equation}\label{eqn_sa}
    \lb \widetilde{\Lom}_{w,h} \, u,v\rb = \iint\limits_{\Om\;\Om}h(x,y) \log \left(\frac{w(x)w(y)}{|x-y|}\right) u(x)v(y)dxdy= \lb \widetilde{\Lom}_{w,h} \,v,u\rb.
\end{equation}
This implies, $\widetilde{\Lom}_{w,h} $ is self-adjoint.
\end{proof}

\begin{remark}\label{remark_continuity}
The kernel $K(x,y)$ of $\widetilde{\mathcal L}_{w,h}$ is
\[
K(x,y)=h(x,y)\,\log\!\left(\frac{w(x)w(y)}{|x-y|}\right),
\]
and it is square integrable as seen below:
\[
\iint\limits_{\Om\;\Om} |K(x,y)|^{2}\,dx\,dy
\le \|h\|_{\infty}
\iint\limits_{\Om\;\Om}
\bigl(|\log(w(x)w(y))| + |\log|x-y||\bigr)^{2} dx\,dy
<+\infty.
\]
Therefore, the operator $\widetilde{\mathcal L}_{w,h}$ can be expressed as a limit of finite rank operators, and hence compact.  However, we established compactness using the Arzelà–Ascoli theorem. The intermediate step \eqref{uniformcontinuity} in our proof implies 
$\widetilde{\mathcal L}_{w,h} \, u\in C(\overline{\Omega})$, for any $u\in L^2(\Om)$. This ensures that the eigenfunctions of $\widetilde{\mathcal L}_{w,h} $ corresponding to non-zero eigenvalues are continuous. 
\end{remark}
Consider the following special case of the above operator, $\Lom_{w,g}$ on $L^2(\Om)$ defined by
    \begin{equation}\label{operator_L_wg}
        \Lom_{w,g}v(x)=\int_\Om \frac{1}{\sqrt{g(x)g(y)}}\log\left(\frac{w(x)w(y)}{|x-y|}\right)v(y)dy,
    \end{equation}
    and the eigenvalue problem:
    \begin{equation}\label{newevproblem}
        \Lom_{w,g}v=\tau v.
    \end{equation}

\noindent \textbf{Proof of Theorem \ref{thm_infinite_ev}:} For $u\in L^2(\Omega)$, set $\tilde{u}=\sqrt{g}u.$ Then, for any $f\in L^2(\Omega)$, observe the following.
$$\inpr{\Lom_w u,f}=\tau \inpr{g u,f} \iff \inpr{\frac{1}{\sqrt{g}}\Lom_w\left(\frac{\tilde{u}}{\sqrt{g}}\right),f}=\tau\inpr{\sqrt{g}u,f} \iff \inpr{\Lom_{w,g}\tilde{u},f}=\tau\inpr{\tilde{u},f}. $$
This shows that  $(\tau,u)$ is an eigenpair of \eqref{evproblem} if and only if $(\tau,\sqrt{g}u)$ is an eigenpair of \eqref{newevproblem}. Moreover, 
\begin{equation}\label{eqn_equiv_innerprod_ortho}
\left.
    \begin{aligned}
\langle \mathcal{L}_{w,g} \tilde{u},\, \tilde{u} \rangle
= \langle \mathcal{L}_{w,g}&(\sqrt{g}u),\, \sqrt{g}u \rangle
= \langle \mathcal{L}_{w} u,\, u \rangle, \\
\int_\Omega \tilde{u}^2 \,dx = \int_\Omega g u^2 \,dx,
\quad &\text{and}\quad
\int_\Omega \tilde{u} \tilde{v} \,dx = \int_\Omega g u v \,dx.
\end{aligned}\quad \right\}
\end{equation}
Taking $h(x,y)=\frac{1}{\sqrt{g(x)g(y)}}$ in Proposition \ref{prop_Lom_compact}, we see that $\Lom_{w,g}$ is compact and self-adjoint. In addition, we have
\begin{equation}
    \Lom_{w,g}v=\left\lb v, \frac{1}{\sqrt{g}}\right\rb\log w +\left\lb v, \frac{\log w}{\sqrt{g}}\right\rb+\Lom v.
\end{equation}
Since $\rm{R}an(\Lom)$ is infinite-dimensional, the above relation shows that $\rm{R}an(\Lom_{w,g})$ is also infinite-dimensional.  Let $(\tau_n)$ be the sequence of all non-zero eigenvalues of $\Lom_{w,g}$ and $\{v_n:n\in \N\}$ be the orthogonal set of eigenfunctions of $\Lom_{w,g}$   given by  Proposition \ref{prop_kesh_func}. Then, for each $n\in \N,$ $u_n=\frac{v_n}{\sqrt{g}}$ is an eigenfunctions of \eqref{evproblem} corresponding to $\tau_n.$ Define $$U_0=L^2(\Omega)\qquad U_n = \left\{u\in L^2(\Omega):\int_\Omega g u u_i=0, \forall\, i=1,2\ldots, n \right\}.$$
Let $\tau_n>0$ be an eigenvalue of $\Lom_{w,g}$. Then by  Proposition \ref{prop_kesh_func} we have 
\begin{equation}
    \tau_n=\frac{\langle \Lom_{w,g} v_n,\, v_n \rangle}{\int_{\Omega} v_n^2 }=\max_{\substack{v \perp  V_{n-1} \\ v \neq 0}}
        \frac{\langle \Lom_{w,g} v,\, v \rangle}{\int_{\Omega} v^2 },
\end{equation}
Now  by \eqref{eqn_equiv_innerprod_ortho}, one can deduce that  
\begin{equation}
    \tau_n=\frac{\langle \Lom_{w} u_n,\, u_n \rangle}{\int_{\Omega} gu_n^2 }=\max_{\substack{ u \in  U_{n-1} \\ u \neq 0}}
        \frac{\langle \Lom_{w} u,\, u \rangle}{\int_{\Omega} gu^2 }.
\end{equation}
An analog variational characterization holds for negative eigenvalues. This completes the proof. \qed

For the operator $\Lom$, the existence of a positive eigenfunction corresponding to $\tau_{1,1}^+(\Om)$ was shown in \cite[Theorem 1.4]{anoopjiya2025} under restriction $\mathrm{diam}(\Omega) \leq 1$. Next, we show a similar result for the eigenfunction corresponding to $\tau_{w,g}^+(\Om)$. The proof follows on the same line as the proof of \cite[Theorem 1.4]{anoopjiya2025}. For completeness, we provide the proof here.
\begin{proposition}\label{prop-non-neg-w}
    Let $\Omega \subset \mathbb{R}^N$ be a bounded open set, and let $w,g\in\cpom$ be such that  $ \sqrt{\mathrm{diam}(\Om)}\leq w$.
 Then, any eigenfunction associated with $\tau_{w,g}^+(\Omega)$ does not vanish in $\Om$.
\end{proposition}

\begin{proof}
     Let $\phi\in\mathcal{M}_g$ be an eigenfunction corresponding to $\tau_{w,g}^+(\Omega)$. Let
    \begin{equation*}
        \Omega^+=\{x\in\Omega: \phi(x)>0\},\;\text{  and  }\;
        \Omega^-=\{x\in\Omega: \phi(x)<0\}.
    \end{equation*}
    If $\phi$ changes sign, then 
    \begin{equation}\label{measnonzero}
        |\Omega^+|\neq 0 \text{ and }|\Omega^-|\neq 0,
    \end{equation}
    and
    \begin{equation}\label{signchange}
        |\phi(x)\phi(y)|>\phi(x)\phi(y),\quad \text{for all } x\in\Omega^+\, ,\, y\in\Omega^-.
    \end{equation}
    From assumption $ \sqrt{\mathrm{diam}(\Om)}\leq w$, we obtain 
    $|x-y|<\mathrm{diam}(\Om)\leq w(x)w(y)\text{ for } x,y\in\Om$ and hence
    \begin{equation}\label{log is positive}
        \log \left(\frac{w(x)w(y)}{|x-y|}\right) > 0,\quad \forall\, x,y\in\Om.
    \end{equation}
    Therefore, from the above inequalities we get
    \begin{align*}
        \lb \Lom_w (|\phi|),|\phi| \rb &=\iint\limits_{\Om^+\;\Om^+}\log\left(\frac{w(x)w(y)}{|x-y|}\right) \phi(x)\phi(y)dxdy + 2\iint\limits_{\Om^+\;\Om^-}\log\left(\frac{w(x)w(y)}{|x-y|}\right) |\phi(x) \phi(y)|dxdy\\
        &\quad\quad\quad\quad + \iint\limits_{\Om^-\;\Om^-}\log\left(\frac{w(x)w(y)}{|x-y|}\right) \phi(x)\phi(y)dxdy\\
        &> \iint\limits_{\Om^+\;\Om^+}\log\left(\frac{w(x)w(y)}{|x-y|}\right) \phi(x)\phi(y)dxdy + 2 \iint\limits_{\Om^+\;\Om^-}\log\left(\frac{w(x)w(y)}{|x-y|}\right) \phi(x)\phi(y)dxdy\\
        &\quad\quad\quad\quad +\iint\limits_{\Om^-\;\Om^-}\log\left(\frac{w(x)w(y)}{|x-y|}\right)\phi(x)\phi(y)dxdy= \lb \Lom_w \phi,\phi \rb=\tau_{w,g}^+(\Om).
    \end{align*}
A contradiction as $\tau_{w,g}^+(\Om)$ as the largest eigenvalue. Thus, $\phi$ cannot change its sign in $\Omega$, and hence $\phi$ can be chosen to be non-negative.   Now suppose that $\phi(x_0)=0$ for some $x_0\in\Om$. Then
    \begin{equation}
        0=\tau_{w,g}^+(\Omega) g(x_0)\, \phi(x_0)=\int_\Omega \log\left(\frac{w(x_0)w(y)}{|x_0-y|}\right) \phi(y)dy.
    \end{equation}
    Thus, from \eqref{log is positive}, we must have $\phi = 0$ a.e. in $\Omega$, a contradiction, as $\phi$ is an eigenfunction. From Remark \ref{remark_continuity}, $\phi$ is continuous, and therefore, $\phi$ does not vanish in $\Omega$.
\end{proof}
\noindent \textbf{Proof of Theorem \ref{theorem_monotonicity_w_g}:} 
\newline $(i)$
        Since $ \sqrt{\mathrm{diam}(\Omega)}\leq w$, by Proposition \ref{prop-non-neg-w}, we get a positive eigenfunction $\phi$ corresponding to $\tau_{w,g}^+(\Omega)$. Now $w \leq w_1$ yields
    \begin{equation}\label{e1}
        \lb \Lom_w \phi,\phi \rb\leq \lb \Lom_{w_1} \phi,\phi \rb.
    \end{equation}
    Since $0<g_1 \leq g$, it follows that 
    \begin{equation}\label{e2}
        \int_\Omega g_1 \phi^{\,2} \, dx\leq \int_\Omega g \phi^{\,2} \, dx.
    \end{equation}
    By the definition of $\tau_{w_1,g_1}^+(\Omega)$ 
    together with \eqref{e1} and \eqref{e2}, we obtain
    \begin{align*}
        \tau_{w,g}^+(\Om)\leq \tau_{w_1,g_1}^+(\Om).
    \end{align*}    
     Furthermore, if $w\neq w_1$, then there exists a subset $\mathcal{N} \subseteq \Omega$ of positive measure such that $w<w_1$ on $\mathcal{N}$. Thus,
    \begin{equation*}
        \frac{w(x)w(y)}{|x-y|}< \frac{w_1(x)w_1(y)}{|x-y|},\quad\text{on }\,\mathcal{N}\times \mathcal{N},
    \end{equation*}
    hence $ \lb \Lom_w \phi,\phi \rb< \lb \Lom_{w_1} \phi,\phi \rb.$  Therefore, 
    \begin{align}\label{e3}
        \begin{split}
        \tau_{w,g}^+(\Om)\leq \tau_{w_1,g_1}^+(\Om).
        \end{split}
        \end{align}
    If $g_1 \neq g$, then by a similar reasoning as before, we deduce that
    \begin{equation*}
        \tau_{w,g}^+(\Om)<\tau_{w_1,g_1}^+(\Om).
    \end{equation*}
    \hfill\\
\noindent $(ii)$    
     Let $\phi\in\mathcal{M}_g$ be an eigenfunction corresponding to $\tau_{w,g}^+(\Om_1)$ and $\widetilde{\phi}$ be its zero extension to $\Om$. Then
   \begin{align*}
       \tau_{w,g}^+(\Om_1)
       &=\lb \Lom_w \widetilde{\phi},\widetilde{\phi} \rb \leq \tau_{w,g}^+(\Om).
   \end{align*}
    In addition, if $\Om_1\subsetneq \Om$, then by Proposition \ref{prop-non-neg-w}, $\widetilde{\phi}$ cannot be an eigenfunction corresponding to $\tau_{w,g}^+(\Om)$. By Proposition \ref{prop_kesh_func}, any extremizer of \eqref{weightedcharmu1} is an eigenfunction, and we obtain
   \begin{align*}
       \lb \Lom_w \widetilde{\phi},\widetilde{\phi} \rb<\tau_{w,g}^+(\Om).
   \end{align*}
    This completes the proof.  \qed
\begin{remark}
    Let $\tau_n$ be an eigenvalue, and let $\Omega_1 \subseteq \Omega$. If $\tau_n>0$, then
    \begin{equation}
        \tau_n(\Om_1)\leq \tau_n(\Om),
    \end{equation}
    whereas if $\tau_n<0$, the monotonicity is reversed,
    \begin{equation}
        \tau_n(\Om_1)\geq \tau_n(\Om).
    \end{equation}
    The proof follows the same argument as in part \ref{item:dom:mon} of the preceding theorem, extending the corresponding eigenfunction by zero to the larger domain $\Omega$.
\end{remark}

In subsequent theorems, we use the notation $\N_0$ to denote $\N \cup \{0\}$.

\noindent \textbf{Proof of Theorem \ref{theorem_convergence}:} For each $n\in\N_0$, consider the potential operator $\mathcal{L}_n$ on $L^2(B_R)$, 
\begin{equation}
    \mathcal{L}_n u(x):=\int_{B_R} l_n(x,y) u(y)dy,
\end{equation}
where
     \begin{equation*}
         l_n(x,y)= \chi_{\Om_n}(x)\chi_{\Om_n}(y)\log \left(\frac{w_n(x)w_n(y)}{|x-y|}\right),\quad x,y\in B_R.
     \end{equation*}
    Observe that $\Lom_n u$ is nothing but the zero extension of $\Lom_{w_n}u$ from $\Omega_n$ to $B_R$. Let $\phi_n$ be an eigenfunction corresponding to $\tau_{w_n,g_n}^+(\Omega_n)$ such that $\int_{\Om_n}g_n\phi_n^2=1$  and $\tilde{\phi}_n$ be its zero extension to $B_R$. It can be seen that
     \begin{align}\label{eqn_ext_ball}
         \tau_{w_n,g_n}^+(\Omega_n)&= \lb \Lom_{w_n} \phi_n,\phi_n \rb=\lb\mathcal{L}_n \tilde{\phi}_n,\tilde{\phi}_n \rb,\quad \forall\, n\in\N_0.
     \end{align}
     Furthermore, for $v\in L^2(B_R)$ with $v|_{\Om_n}\neq 0$, we consider a test function  $u=v|_{\Om_n}\left(\int_{\Om_n}gv^2\right)^{-\frac{1}{2}}$ in $L^2(\Om_n)$ to deduce that
     \begin{align}\label{tau_def_in_ball}
         \tau_{w_n,g_n}^+(\Omega_n)
         &=\max\left\{ \lb \mathcal{L}_n v,v \rb_{L^2(B_R)}: v\in L^2(B_R) \text{ and } \int_{B_R}g_nv^2= 1 \right\}.
     \end{align}
       Therefore, for $v=\tilde{\phi}_k\left(\int_{B_R}g_j(\tilde{\phi}_{k})^2 \right)^{-\frac{1}{2}}$ in $L^2(B_R),$ we get
       \begin{equation}\label{Lmn_combo}
           \lb \mathcal{L}_j \tilde{\phi}_k,\tilde{\phi}_k \rb\leq \tau_{w_j,g_j}^+(\Omega_n)\int_{B_R}g_j 
         (\tilde{\phi}_k)^2,\quad\forall\, j,k\in\N_0.
       \end{equation}
     Hence, from \eqref{eqn_ext_ball} and \eqref{Lmn_combo} (with  $j=0$ and $k=n$), we obtain
     \begin{align}\label{new_e1}
         \lb(\mathcal{L}_n-\mathcal{L}_0)\tilde{\phi}_n,\tilde{\phi}_n \rb&\geq \tau_{w_n,g_n}^+(\Omega_n)- \tau_{w_0,g_0}^+(\Omega_0)\int_{B_R}g_0 
         (\tilde{\phi}_n)^2.
     \end{align}
     Similarly, from \eqref{eqn_ext_ball} (with $n=0$), and \eqref{Lmn_combo} (with $j=n$ and $k=0$), we get
     \begin{align}\label{new_e2}
         \lb(\mathcal{L}_n-\mathcal{L}_0)\tilde{\phi}_0,\tilde{\phi}_0 \rb&\leq \tau_{w_n,g_n}^+(\Omega_n)\int_{B_R}g_n 
         (\tilde{\phi}_0)^2- \tau_{w_0,g_0}^+(\Omega_0).
     \end{align}
      Since $g_n\to g_0$ uniformly, for sufficiently small  $\epsilon>0$, there exists  $n_0\in\N$  such that
      \begin{equation}\label{g_n_bounds}
          0<g_0-\epsilon\leq g_n\leq g_0+\epsilon,\quad \forall \, n\geq n_0,
      \end{equation}
     hence, there exists $m>0$ such that
      $$m< g_n,\forall\, n\in \N_0.$$ Consequently, since $\int_{B_R} g_n(\tilde{\phi}_n)^2=1$,  we obtain
      \begin{equation}
          \int_{B_R} (\tilde{\phi}_n)^2< \frac{1}{m}, \forall\, n\in \N_0. 
      \end{equation}
      Thus, by \eqref{g_n_bounds}, for every $n\geq n_0$, we get
      \begin{align}
         \int_{B_R}g_0 (\tilde{\phi}_n)^2 &\leq \int_{B_R}g_n (\tilde{\phi}_n)^2+\epsilon \int_{B_R} (\tilde{\phi}_n)^2< 1+\frac{\epsilon}{m}, \\
          \int_{B_R}g_n (\tilde{\phi}_0)^2 & \leq \int_{B_R}g_0 (\tilde{\phi}_0)^2+\epsilon \int_{B_R} (\tilde{\phi}_0)^2< 1+\frac{\epsilon}{m}.
      \end{align}
      Therefore, from \eqref{new_e1} and \eqref{new_e2}, $\forall \, n\geq n_0,$ we get 
      \begin{align}\label{combinedinequality}
         \lb(\mathcal{L}_n-\mathcal{L}_0)\tilde{\phi}_0,\tilde{\phi}_0 \rb-\tau_{w_n,g_n}^+(\Omega_n)\frac{\epsilon}{m}
         &\leq \tau_{w_n,g_n}^+(\Omega_n)-\tau_{w_0,g_0}^+(\Omega_0)\nonumber\\
        &\leq \lb(\mathcal{L}_n-\mathcal{L}_0)\tilde{\phi}_n,\tilde{\phi}_n \rb+\tau_{w_0,g_0}^+(\Omega_0)\frac{\epsilon}{m},
     \end{align}      
     Since $w_n\to w_0$ uniformly, there exists a $C>0$ such that
     \begin{equation}\label{Lnbound}
         |l_n(x,y)|\leq C+\left|\log\frac{1}{|x-y|}\right|, \quad\forall\, n\in\N_0.
     \end{equation}
     Consequently, we deduce that
     \begin{align}
         \tau_{w_n,g_n}^+(\Omega_n)&=\lb\mathcal{L}_n \tilde{\phi}_n,\tilde{\phi}_n \rb\leq \frac{1}{m} \left(\iint\limits_{B_R\;B_R} |l_n(x,y)|^2dxdy\right)^{\frac{1}{2}}\\
         &\leq \frac{2}{m}  \left(\iint\limits_{B_R\;B_R} \left[C^2+\left|\log\frac{1}{|x-y|}\right|^2\right]dxdy\right)^{\frac{1}{2}}=:C',\;\forall \, n\in \N_0.
     \end{align}  
    Thus, from \eqref{combinedinequality}, we obtain
     \begin{align*}
         |\tau_{w_n,g_n}^+(\Omega_n)-\tau_{w_0,g_0}^+(\Omega_0)|
         &\leq \frac{1}{m} \left(\iint\limits_{B_R\;B_R} |l_n(x,y)-l_0(x,y)|^2dxdy\right)^{\frac{1}{2}}+\frac{ C'}{m}\epsilon,\,\;\forall \, n\geq n_0.
     \end{align*}     
     From the convergence of $(\chi_{\Om_n})$ and $(w_n)$,  it follows that $l_n(x,y) \to l_0(x,y)$ pointwise. Now, \eqref{Lnbound} enables us to use the dominated convergence theorem and conclude 
     \begin{equation*}
         \iint\limits_{B_R\;B_R} |l_n(x,y)-l_0(x,y)|^2dxdy \to 0.
     \end{equation*}    
     Hence, $\tau_{w_n,g_n}^+(\Omega_n)\to \tau_{w_0,g_0}^+(\Om_0)$ as required. \qed 

\noindent \textbf{Proof of Theorem \ref{theorem_convergence_negative}:} The proof follows exactly the same steps as in the proof of Theorem \ref{theorem_convergence}. Let $\psi_n$ be an eigenfunction corresponding to $\tau_{w_n,g_n}^-(\Omega_n)$ such that $\int_{\Om_n}g_n\psi_n^2=1$  and $\tilde{\psi}_n$ be its zero extension to $B_R$. As before, we have
     \begin{align}\label{eqn_ext_ball_neg}
         \tau_{w_n,g_n}^-(\Omega_n)&= \lb \Lom_{w_n} \psi_n,\psi_n \rb_{L^2(\Om_n)}=\lb\mathcal{L}_n \tilde{\psi}_n,\tilde{\psi}_n \rb_{L^2(B_R)},\quad \forall\, n\in\N_0,
     \end{align}
     and
     \begin{align}\label{tau_def_in_ball_neg}
         \tau_{w_n,g_n}^-(\Omega_n)  
         &=\inf\left\{ \lb \mathcal{L}_n v,v \rb_{L^2(B_R)}: v\in L^2(B_R) \text{ and } \int_{B_R}g_nv^2= 1 \right\}.
     \end{align}
       Hence, for $v = \tilde{\psi}_k \left(\int_{B_R} g_j (\tilde{\psi}_k)^2 \right)^{-1/2}$ in $L^2(B_R)$, we obtain
       \begin{equation}\label{Lmn_combo_neg}
           \lb \mathcal{L}_j \tilde{\psi}_k,\tilde{\psi}_k \rb\geq \tau_{w_j,g_j}^-(\Omega_j)\int_{B_R}g_j 
         (\tilde{\psi}_k)^2,\quad\forall\, j,k\in\N_0.
       \end{equation}  
      Proceeding analogously to \eqref{combinedinequality}, for $ n\geq n_0$, we obtain
      \begin{align}\label{combinedinequality_neg}
         \lb(\mathcal{L}_n-\mathcal{L}_0)\tilde{\psi}_n,\tilde{\psi}_n \rb_{L^2(B_R)}+\tau_{w_0,g_0}^-(\Omega_0)\frac{\epsilon}{m}         
         &\leq \tau_{w_n,g_n}^-(\Omega_n)-\tau_{w_0,g_0}^-(\Omega_0)\nonumber\\
        &\leq\lb(\mathcal{L}_n-\mathcal{L}_0)\tilde{\psi}_0,\tilde{\psi}_0 \rb_{L^2(B_R)}-\tau_{w_n,g_n}^-(\Omega_n)\frac{\epsilon}{m}.
     \end{align}  
    Repetition of the estimates established after \eqref{combinedinequality} in the preceding theorem leads to the conclusion that
 $\tau_{w_n,g_n}^-(\Omega_n)\to \tau_{w_0,g_0}^-(\Om_0)$.  \qed     

 To establish the reverse Faber-Krahn inequality for $\tau_{w,g}^+$, we require a Riesz-type inequality for functions in the following function space: 
\begin{equation*}
    \mathcal{F}_w=\left\{f:f \text{ is measurable on }\R^N \text{ and } \iint\limits_{\R^N\;\R^N} \left|\log \left(\frac{w(x)w(y)}{|x-y|}\right)\right|\big| f(x)f(y) \big| dxdy<\infty\right\}.
\end{equation*}
The case $w=1$ was proved in \cite[Proposition~2.6]{anoopjiya2025}. In the following proposition, we consider a more general $w$. The proof follows the same approach outlined in \cite[Proposition 2.6]{anoopjiya2025}, with some slight modifications. For completeness, we provide a proof here. 
\begin{proposition}\label{riezpol}
    Let $H$ be a polarizer, and let $f$ be a measurable function on $\mathbb{R}^N$, such that $f,P_H(f)\in\mathcal{F}_w$. Assume that $w(x)=w(\sigma_H(x))$, for all $x\in\R^N$. Then
    \begin{equation}\label{riez_type_inequality_pol}
        \iint\limits_{\R^N\;\R^N} \log \left(\frac{w(x)w(y)}{|x-y|}\right) f(x)f(y) dxdy \leq  \iint\limits_{\R^N\;\R^N} \log \left(\frac{w(x)w(y)}{|x-y|}\right) P_H(f)(x)P_H(f)(y) dxdy .
    \end{equation}
    In addition, the equality holds in \eqref{riez_type_inequality_pol} only if either $ P_H(f)=f$ a.e. or $P_H(f)=f\circ \sigma_H $ a.e.
\end{proposition}
\proof 
     Define 
    \begin{equation*}
        I(f)= \iint\limits_{\R^N\;\R^N} \log \left(\frac{w(x)w(y)}{|x-y|}\right) f(x)f(y) dxdy .
     \end{equation*}
 Denote $\overline{x}=\sigma_H(x)$ and $K(x,y)=\log \left(\frac{w(x)w(y)}{|x-y|}\right)$. Since $w(x)=w(\ol{x})$, by applying a change of variable, we get:
 \begin{equation}
     I(f)=\iint\limits_{H\;H} \big[ K(x,y)(f(x)f(y)+f(\overline{x})f(\overline{y}))+K(x,\overline{y})(f(x)f(\overline{y})+f(\overline{x})f(y)) \big] dxdy. \\
 \end{equation}
 Let 
 \begin{equation*}
     \mathcal{S}_{f}(x,y):=K(x,y)(f(x)f(y)+f(\overline{x})f(\overline{y}))+K(x,\overline{y})(f(x)f(\overline{y})+f(\overline{x})f(y)).
 \end{equation*}
 Then
  \begin{equation}\label{I_f}
     I(f)= \iint\limits_{A\;A} \mathcal{S}_{f}(x,y) dxdy + 2 \iint\limits_{B\;A} \mathcal{S}_{f}(x,y) dxdy +\iint\limits_{B\;B} \mathcal{S}_{f}(x,y) dxdy,
 \end{equation}
 where  
 \begin{equation*}
     A=\{x\in H: f(x)\geq f(\overline{x})\}\text{ and }B=\{x\in H: f(x)< f(\overline{x})\} .
 \end{equation*}
 Notice that
 \begin{equation}\label{x_in A_x_in B def}
\begin{aligned}
  P_H(f)(x)=f(x) &\quad\text{and}\quad P_H(f)(\overline{x})=f(\overline{x}),\quad\forall\, x\in A,\\
     P_H(f)(x)=f(\overline{x}) &\quad\text{and}\quad  P_H(f)(\overline{x})=f(x),\quad\forall\, x\in B.
\end{aligned}
\end{equation}
It is easy to verify that $$\mathcal{S}_{f}(x,y)=\mathcal{S}_{P_H(f)}(x,y), \forall\, (x,y)\in \left(A\times A\right)\, \bigcup\, \left(B\times B\right).$$ Consequently,
\begin{equation}\label{A_A}
    \iint\limits_{A\;A} \mathcal{S}_{f}(x,y) dxdy= \iint\limits_{A\;A} \mathcal{S}_{P_H(f)}(x,y) dxdy,
\end{equation}
and
\begin{equation}\label{B_B}
    \iint\limits_{B\;B} \mathcal{S}_{f}(x,y) dxdy= \iint\limits_{B\;B} \mathcal{S}_{P_H(f)}(x,y) dxdy.
\end{equation}
On the other hand, for $(x,y)\in A\times B$,       
    \begin{equation}\label{S_PHf-S_f}
        \mathcal{S}_{P_H(f)}(x,y)-\mathcal{S}_{f}(x,y)=(f(x)-f(\overline{x}))(f(\overline{y})-f(y))(K(x,y)-K(x,\overline{y})).
    \end{equation}
     Since  $|x-y|< |x-\overline{y}|$(\cite[Proposition 2.5]{anoopjiya2025}) and $w(y)=w(\ol{y})$, it follows that $K(x,y)>K(x,\overline{y})$ and hence
    \begin{equation*}
        \mathcal{S}_{P_H(f)}(x,y)-\mathcal{S}_{f}(x,y)\geq 0.
    \end{equation*}
    Therefore,
    \begin{equation}\label{B_A}
    \iint\limits_{B\;A} \mathcal{S}_{f}(x,y) dxdy\leq \iint\limits_{B\;A} \mathcal{S}_{P_H(f)}(x,y) dxdy.
\end{equation}
Now, by combining \eqref{A_A}, \eqref{B_B}and \eqref{B_A}, we conclude  $I(f)\leq I(P_H(f))$. 

If  $I(f)=I(P_H(f))$, then
\begin{equation*}
    \iint\limits_{B\;A} \mathcal{S}_{f}(x,y)dxdy=\iint\limits_{B\;A} \mathcal{S}_{P_H(f)}(x,y)dxdy.
\end{equation*}
Therefore, $$\mathcal{S}_{P_H(f)}(x,y)=\mathcal{S}_{f}(x,y) \text{ a.e on } A\times B.$$ Let $A_1:=\{x\in A: f(x)>f(\overline{x})\}.$ 
Since $K(x,y)>K(x,\overline{y})$,  from \eqref{S_PHf-S_f} we conclude that $\mathcal{S}_{P_H(f)}(x,y)>\mathcal{S}_{f}(x,y)$, for $(x,y)\in A_1\times B$. Therefore, $|A_1\times B|=0$, and therefore $|A_1|=0$ or $|B|=0$. Notice that
\begin{enumerate}[(i)]
    \item if $|B|=0$, then $f(x)\geq f(\overline{x})$  a.e. in $H$, and hence $P_H(f)=f$ a.e. in $\R^N,$
     \item if $|A_1|=0$, then $f(x)\leq f(\overline{x})$ a.e. in $H$, and hence $P_H(f)=f\circ\sigma_H$  a.e. in $\R^N$.  
\end{enumerate}
 This concludes the proof. \qed
 
 \noindent {\bf Proof of Theorem \ref{weightedfk}:} Let $\phi_1$ be the eigenfunction corresponding to $\tau_{w,g}^+(\Om)$, satisfying $\int_\Om g \phi_1^2=1$ and $\phi_1>0$ on $\Om\setminus\sigma_H(\Om)$. We also denote the zero extension of  $\phi_1$ to $\R^N$ by $\phi_1$. Consider the sets
\begin{equation*}
     A=\{x\in H: \phi_1(x)\geq \phi_1(\overline{x})\}\text{ and }B=\{x\in H: \phi_1(x)< \phi_1(\overline{x})\} .
 \end{equation*}
 By the definition of polarization, we have:
 \begin{equation}\label{x_in A_x_in B}
\begin{aligned}
  P_H(\phi_1)(x)=\phi_1(x) &\quad\text{and}\quad P_H(\phi_1)(\overline{x})=\phi_1(\overline{x}),\quad\forall\, x\in A,\\
     P_H(\phi_1)(x)=\phi_1(\overline{x}) &\quad\text{and}\quad  P_H(\phi_1)(\overline{x})=\phi_1(x),\quad\forall\, x\in B.
\end{aligned}
\end{equation}
In both cases, it follows that
\begin{equation*}
    P_H (\phi_1)(x)^2+P_H (\phi_1)(\ol{x})^2=\phi_1(x)^2+\phi_1(\ol{x})^2.
\end{equation*}
Therefore,
\begin{align*}
        \int_{\mathbb{R}^2} g(x) P_H(\phi_1)(x)^2 \, dx 
        &= \int_H g(x) P_H(\phi_1)(x)^2 \, dx + \int_{H^c} g(x) P_H(\phi_1)(x)^2 \, dx \\
        &= \int_H g(x) \left[ P_H(\phi_1)(x)^2 + P_H(\phi_1)(\overline{x})^2 \right] dx \quad \text{(since $g(x) = g(\overline{x})$)} \\
        &= \int_H g(x) \left[ \phi_1(x)^2 + \phi_1(\overline{x})^2 \right] dx \\
        &= \int_{\mathbb{R}^2} g(x) \phi_1(x)^2 \, dx \quad \text{(again using symmetry of $g$)}.
\end{align*}
Since $\phi_1 > 0$ in $\Omega\setminus\sigma_H(\Omega)$ and is extended by zero outside, it follows that $P_H(\phi_1) = 0$ on $P_H(\Omega)^c$, and hence,
\begin{equation*}
    \int_{P_H(\Omega)} g P_H(\phi_1)^2 = \int_\Omega g (\phi_1)^{\, 2} = 1.
\end{equation*}
Therefore, using \eqref{riez_type_inequality_pol} and the definition of $\tau_{w,g}^+(P_H(\Omega))$, we obtain
\begin{align}\label{logpolid}
        \tau_{w,g}^+(\Omega)&=\iint\limits_{\Om\;\Om}\log\left(\frac{w(x)w(y)}{|x-y|}\right)\phi_1(x)\phi_1(y) dxdy\nonumber\\
        &\leq  \int\limits_{P_H(\Om)} \int\limits_{P_H(\Om)} \log\left(\frac{w(x)w(y)}{|x-y|}\right)P_H(\phi_1)(x)P_H(\phi_1)(y)dxdy\leq \tau_{w,g}^+(P_H(\Omega)).\nonumber
    \end{align} 
    The equality case follows the same argument as in the proof of \cite[Theorem 1.6]{anoopjiya2025}.  \qed
    
    The following remark is required for proving Theorem \ref{theorem_weighted_trans}
\begin{remark}\label{rem:superharmonic}
    For a fixed $y\in\mathbb{R}^N$ with $N\geq 2$, the function $f(x)=\log \frac{1}{|x-y|}$ satisfies   
    \begin{equation}
        -\Delta_x f(x)=\frac{(N-2)}{|x-y|^2}\geq 0, \quad \forall x\in \R^N\setminus\{y\}.
    \end{equation}
 Thus, $f$ is superharmonic on $\R^N\setminus\{y\}$. Consequently, $f$ satisfies the mean value inequality 
    \begin{equation}\label{eqn:meanvalueinequality}
            \int_{B_R(a)} \log \frac{1}{|x-y|} dx\leq |B_R(a)|\log \frac{1}{|a-y|},\quad a\neq y.
    \end{equation}
\end{remark}

\noindent \textbf{Proof of Theorem \ref{theorem_weighted_trans}:} Define 
    \begin{equation}
        \overline{\Om}_\epsilon:=\{x\in \overline{\Om}:dist(x,\partial \Om)\geq 2\epsilon \}.
    \end{equation}
     Taking $W = \frac{1}{w}$ in Proposition \ref{thm:weighted_diam_limit},  we have $T_{diam}\left(\overline{\Om}_\epsilon,\frac{1}{w}\right)\to T_{diam}\left(\overline{\Om},\frac{1}{w}\right)$ as $\epsilon \to 0$. Thus, we can choose $\epsilon_0 > 0$ sufficiently small so that $T_{diam}(\overline{\Om}_{\epsilon_0},W) > 1$. For each  $n \in \N$, let $x_1, x_2, \dots, x_n$ be a set of weighted Fekete points corresponding to $\rho_n\left(\overline{\Om}_{\epsilon_0},\frac{1}{w}\right)$, that is 
     $$\rho_n\left(\overline{\Om}_{\epsilon_0},\frac{1}{w}\right)= \left [\underset{1\leq i<j\leq n}{\prod} \frac{|x_i-x_j|}{w(x_i)w(x_j)} \right]^{\frac{2}{n(n-1)}}. $$
     Since ${\rho_n\left(\overline{\Om}_{\epsilon_0},\frac{1}{w}\right)}$ is a decreasing sequence with respect to $n$\cite[Theorem 1.1, p.143]{ebsaff1997}, it follows that $\rho_n\left(\overline{\Om}_{\epsilon_0},\frac{1}{w}\right) > 1$, for every $n\in\N$. Let 
    \begin{align}\label{r_N_K_N_f_N}
        r_n &= C_N\frac{1}{\sqrt[N]{n}}, \quad C_N=\left[\frac{\Gamma\left(\frac{N}{2}+1\right)}{\pi^{N/2}}\right]^{\frac{1}{N}},\\
    B_{n_j} &=B_{r_n}(x_j)=\{x : |x - x_j| \leq r_n\}, \quad j = 1, 2, \dots, n,\\
            f_{n}(x) &= \sum_{j=1}^n \chi_{B_{n_j}},\quad n\in\N.
    \end{align}
    For $n>\left(\frac{C_N}{ \epsilon_0}\right)^N$, we have $B_{n_j}\subset \Omega,$ and the function $f_{n}\in L^2(\Omega).$ Moreover, 
    \begin{align}\label{Inu_n2}
        \langle \Lom_w f_{n}, f_{n} \rangle&= \iint\limits_{\Om\;\Om} \log \left(\frac{w(x)w(y)}{|x-y|}\right)f_{n}(x)f_{n}(y)dxdy
        \nonumber\\
        &=\sum_{j=1}^n\,\,\, \iint\limits_{B_{n_j}\;B_{n_j}} \log \left(\frac{w(x)w(y)}{|x-y|}\right) dxdy+\sum_{\substack{i,j=1\\i\neq j}}^n \,\,\,\iint\limits_{B_{n_i}\;B_{n_j}} \log \left(\frac{w(x)w(y)}{|x-y|}\right) dxdy.\nonumber\\
    \end{align}
        For a fixed $x$, the function $\log\frac{1}{|x-y|}$ is superharmonic in $y$, and from Remark \ref{rem:superharmonic}, it follows that 
        \begin{equation}
            \int_{B_{n_j}} \log \frac{1}{|x-y|} dx\leq |B_{n_j}|\log \frac{1}{|x_j-y|},\quad y\neq x_j,
        \end{equation}
        where $|B_{n_j}|$ denotes the Lebesgue measure of $B_{n_j}$. Hence, for $i=j$, we obtain
        \begin{align}
            \iint\limits_{B_{n_j}\;B_{n_j}} \log \frac{1}{|x-y|} dxdy&\leq |B_{n_j}|\int_{B_{n_j}}\log \frac{1}{|x_j-y|} dy=\frac{1}{n}\int_{B_{r_n}}\log\frac{1}{|y|}dy\nonumber\\
            &=\frac{1}{n}|\mathbb{S}^{N-1}|\left[\frac{(r_n)^N}{N}\log\frac{1}{r_n}+\frac{ (r_n)^N}{N^2}\right].
        \end{align}
        Let $M\geq w$. Then from \eqref{r_N_K_N_f_N} and the above inequality, we have
        \begin{equation}
            \iint\limits_{B_{n_j}\;B_{n_j}} \log \left(\frac{w(x)w(y)}{|x-y|}\right) dxdy\leq \frac{1}{n}\frac{1}{N}|\mathbb{S}^{N-1}|\left[(r_n)^N \log\frac{1}{r_n}+\frac{ (r_n)^N}{N}\right] + \frac{1}{n^2}\log M^2.
        \end{equation} 
        Thus,
        \begin{equation}\label{app_eq2}
            \sum_{j=1}^n\,\,\, \iint\limits_{B_{n_j}\;B_{n_j}} \log \left(\frac{w(x)w(y)}{|x-y|}\right) dxdy\leq\frac{1}{N}|\mathbb{S}^{N-1}|\left[(r_n)^N \log\frac{1}{r_n}+\frac{ (r_n)^N}{N}\right] + \frac{1}{n}\log M^2.
        \end{equation}  
         Let $i \neq j$, and let $S=\log w$. Since $S$ is superharmonic, we have
        \begin{align}
            \iint\limits_{B_{n_i}\;B_{n_j}} \log \left(\frac{w(x)w(y)}{|x-y|}\right) dxdy &=
            \iint\limits_{B_{n_i}\;B_{n_j}} \left(S(x)+S(y)+\log\frac{1}{|x-y|}\right)dxdy\nonumber\\
            &\leq |B_{n_j}| \int_{B_{n_i}} \left(S(x_j)+S(y)+\log\frac{1}{|x_j-y|}\right)dy\nonumber \\
            &\leq |B_{n_j}| |B_{n_i}|\left(S(x_j)+S(x_i)+\log \frac{1}{|x_i-x_j|}\right)\nonumber\\
            &=\frac{1}{n^2}\log \left(\frac{w(x_i)w(x_j)}{|x_i-x_j|}\right).
        \end{align}
       Taking the sum over all terms, we obtain
        \begin{align}\label{app_eq1}
            \sum_{\substack{i,j=1\\i\neq j}}^n \,\,\,&\iint\limits_{B_{n_i}\;B_{n_j}} \log \left(\frac{w(x)w(y)}{|x-y|}\right) dxdy\leq \frac{1}{n^2}\sum_{\substack{i,j=1\\i\neq j}}^n\log \left(\frac{w(x_i)w(x_j)}{|x_i-x_j|}\right)\nonumber\\         
            &=\frac{1}{n^2} n (n-1)\log \frac{1}{\rho_n\left(\overline{\Om}_{\epsilon_0},\frac{1}{w}\right)}=\frac{n-1}{n} \log \frac{1}{\rho_n\left(\overline{\Om}_{\epsilon_0},\frac{1}{w}\right)},
        \end{align}
        where $\rho_n\left(\overline{\Om}_{\epsilon_0},\frac{1}{w}\right)= \left [\underset{1\leq i<j\leq n}{\prod} \frac{|x_i-x_j|}{w(x_i)w(x_j)} \right]^{\frac{2}{n(n-1)}}>1. $
         Substituting \eqref{app_eq2} and \eqref{app_eq1} in \eqref{Inu_n2} leads to 
        \begin{align}
             \langle \Lom_w f_{n}, f_{n}\rangle\leq \frac{1}{N}|\mathbb{S}^{N-1}|\left[(r_n)^N \log\frac{1}{r_n}+\frac{ (r_n)^N}{N}\right]+ \frac{1}{n}\log M^2+\frac{n-1}{n} \log \frac{1}{\rho_n\left(\overline{\Om}_{\epsilon_0},\frac{1}{w}\right)}<0,
        \end{align}
        for sufficiently large $n$. Therefore, $\Lom_w$ is not a positive operator on $L^2(\Om)$ and $\tau_{w,g}^-(\Om)<0$. \qed
       
 The next proposition shows that, if $\Lom_w$ is a positive operator, then $\tau_{w,g}^-(\Om)=0$.
\begin{proposition}\label{prop_tau_tilde_zero}
    Let $\Om\subset \R^N$ be a bounded open set, and $w\in\cpom$. Then there exists a sequence $(u_n)\text{ in }\mathcal{M}_g$ such that $\lb \Lom_w u_n, u_n\rb\to 0$. Consequently, 
    if $\Lom_w$ is a positive operator on $L^2(\Om)$, then $\tau_{w,g}^-(\Om)=0$.
\end{proposition}
\begin{proof}
    Let $x=(x_1,x_2,\cdots,x_N)\in\R^N$. Define $v_n(x)=\sin (nx_1)$.
 By the Riemann-Lebesgue lemma\cite[p.94]{SteinShakarchi},  for every $f\in L^1(\R^N)$(and hence for every $f\in L^2(\Om)$), we have
    \begin{equation}\label{eqn_riemann_lebesgue}
        \int_{\Om} f(x) \sin (nx_1)dx\to 0,\; \text{and}\; \int_{\Om} f(x) \cos (nx_1)dx\to 0, \quad\text{as } n\to\infty.
    \end{equation}
     Thus, $v_n 	\xrightharpoonup{w} 0$. Since $\Lom_w$ is a compact operator and $v_n$ is a bounded sequence, we conclude 
    \begin{equation}
        |\lb \Lom_w v_n, v_n\rb|\leq \|\Lom_w v_n\|\|v_n\|\to 0.
    \end{equation}     
     Moreover, 
    \begin{align}
        \int_{\Om} g(x) v_n(x)^2dx&= \int_{\Om} g(x) (\sin (nx_1))^2dx\geq m \int_{\Om} \sin^2 (nx_1)dx\nonumber\\
        &=\frac{m}{2}\int_{\Om}(1-\cos (nx_1))dx=\frac{m|\Om|}{2}-\frac{m}{2}\int_{\Om}\cos (nx_1)dx.
    \end{align}
    By \eqref{eqn_riemann_lebesgue}, there exists $n_0\in\N$ such that $\int_{\Om} g(x) v_n(x)^2dx\geq C>0, \;\forall\, n\geq n_0$. Now, define  $$u_n(x)=\frac{v_n(x)}{\sqrt{\int_\Om g(x) v_n(x)^2 dx}}.$$ 
    It is easy to see that $(u_n)$ satisfies all requirements. 
\end{proof}

\section{On the planar case}\label{sec_twodim}
In this section, we prove results that are specific to the planar case. We discuss the possibility of \(0\) as an eigenvalue in Theorem~\ref{theorem_0_ev}, and the representation of eigenfunctions in Theorem \ref{efreptheorem}. Before proving Theorem~\ref{theorem_0_ev}, we first prove the following proposition.
\begin{proposition}\label{prop:Kernel}
    Let $\Omega \subset \mathbb{R}^2$ be a bounded open set, and let $w\in\cpom$ be such that  $\Delta (\log w )=C$, for some constant $ C.$  Then any function $u$ satisfying $\Lom_w u=0$ must be constant.  
\end{proposition}
\begin{proof}
   Let $u$ be  a function satisfying $\Lom_w u=0$. Then, by expanding $log$ in \eqref{def_lw}, we get
    \begin{equation}\label{lap_Q_is_C}
        \Lom_w u(x) = \log w(x) \int_\Omega u(y)\,dy + \!\int_\Omega \log w(y)u(y)\,dy +  \Lom u(x)=0.
    \end{equation}
   Now, we apply the Laplacian on both sides of the above equation to get
   \begin{equation}
       \Delta\log w(x) \int_\Omega u(y)\,dy +\Delta  \Lom u(x)=0,
   \end{equation}
   and use the fact that $\Delta \Lom u = -2\pi u$\cite[Lemma 4.2]{gilbarg_trudinger} to obtain
   \begin{equation}\label{ef_ev0_is_constant}
       u(x) =\frac{\Delta \log w(x)}{2\pi} \!\int_\Omega u(y)\,dy= \frac{C}{2\pi} \!\int_\Omega u(y)\,dy.
   \end{equation}
   Thus, $u$ must be a constant. 
\end{proof}
\begin{remark}
     From \eqref{lap_Q_is_C}, one can see that $\Lom$ maps ${\rm Null}(\Lom_w)$ to $\operatorname{span}\{\log w,\, 1\}$. Since $\Lom$ is injective, it follows that $\dim {\rm Null}(\Lom_w) \le 2.$ Furthermore, from \eqref{operator_L_wg}, we have 
     \begin{equation}
         \Lom_{w,g}(\sqrt{g}u)=\frac{1}{\sqrt{g}}\Lom_w u,
     \end{equation}
     and hence
    \begin{equation}
        {\rm Null}(\Lom_{w,g})=\{\sqrt{g}u:u\in {\rm Null}(\Lom_w)\}.
    \end{equation}
    Therefore, we also have $\dim {\rm Null}(\Lom_{w,g}) \le 2.$    
\end{remark}

\noindent \textbf{Proof of Theorem \ref{theorem_0_ev}:} Suppose that $0$ is an eigenvalue of \eqref{evproblem} with the corresponding eigenfunction $u$. By the previous proposition, $u$ must be constant and by  \eqref{ef_ev0_is_constant}, $$u(x)=\frac{C}{2\pi}u(x)|\Om|,\quad\forall\,x\in\Om.$$ 
    However, since $C \neq \frac{2\pi}{|\Omega|}$, we obtain $u(x) = 0$ for all $x \in \Omega$. Thus, $0$ is not an eigenvalue of \eqref{evproblem}.  \qed

 In the next proposition, we compute $\Lom(u)$, for $u=1$.
\begin{proposition}
    Let $\Lom$ be the logarithmic potential operator defined on $L^2(B_R)$.
    \begin{equation}\label{term3inLw1}
        \Lom(1)(x) =\int_{B_R}\log \frac{1}{|x-y|}dy= -\frac{\pi |x|^2}{2} - \pi R^2\log R + \frac{\pi R^2}{2}.
    \end{equation}
\end{proposition}
\begin{proof}
    First, we rewrite the variables in polar coordinates as $x=re^{i\theta}$ and $y=\rho e^{i\omega} $. Thus,
    \begin{align}
        \Lom (1)(x)&=\int_{B_R}\log \frac{1}{|x-y|}dy= \int_0^R \int_0^{2\pi} \log \frac{1}{|re^{i\theta}-\rho e^{i\omega}|}d\omega \rho d\rho\\
        &= \int_0^R \int_0^{2\pi} \log \frac{1}{|r-\rho e^{i\omega}|}d\omega \rho d\rho.
    \end{align}
    Therefore, $\Lom (1)(x)=\Lom (1)(|x|)$. It is known (see \cite[Example 5.7, p. 22]{ebsaff1997})
    \begin{equation}
        \frac{1}{2\pi} \int_0^{2\pi} \log \frac{1}{|r-\rho e^{i\omega}|}d\omega= \begin{cases}
            \log \frac{1}{r}, & \rho\leq r,\\
            \log \frac{1}{\rho}, & \rho\geq r.
        \end{cases}
    \end{equation}
    \noindent Thus,
    \begin{align}
        \Lom (1)(x)&=2\pi\left[\int_0^r \log \frac{1}{r} \rho d\rho +\int_r^R \log \frac{1}{\rho} \rho d\rho\right] =2\pi\left[\frac{r^2}{2}\log \frac{1}{r}+\frac{R^2}{2}\log \frac{1}{R}+\frac{R^2}{4}-\frac{r^2}{2}\log \frac{1}{r}-\frac{r^2}{4}\right]\nonumber\\
        &=2\pi\left[-\frac{|x|^2}{4}-\frac{R^2}{2}\log R+\frac{R^2}{4}\right].
    \end{align}
\end{proof}

Next, we provide weight functions $w$ that satisfy  $\Delta \log w = \frac{2\pi}{|\Omega|}$.
\begin{example}\label{example_disk}
    For $R>0$, let 
    \begin{equation*}
        w_1(x):=e^{\frac{|x|^2}{2R^2}+\frac{1}{2}\log R-\frac{3}{8}}, \forall\, x\in B_R,
    \end{equation*}
    and
    $$Q_c(x):=log\, (cw_1)(x)=\frac{|x|^2}{2R^2}+\frac{1}{2}\log R-\frac{3}{8}+\log c,\quad c>0.$$ 
     Then, for each $c>0$, $\Delta Q_c = \frac{2}{R^2}=\frac{2\pi}{|B_R|}$, and 
    \begin{equation}\label{LW1is0}
        \Lom_{cw_1} (1)(x)= Q_c(x) \int_{B_R} \,dy + \!\int_{B_R} Q_c(y)\,dy + \Lom (1)(x).
    \end{equation}
    Next, we compute each term in \eqref{LW1is0}: 
    \begin{equation}\label{term1inLw1}
        Q_c(x) \int_{B_R} \,dy= \frac{\pi |x|^2}{2}+\frac{\pi R^2}{2}\log R-\frac{3}{8}\pi R^2+\pi R^2 \log c,
    \end{equation}
    and
    \begin{align}\label{term2inLw1}
       \int_{B_R}  Q_c(y)  \,dy&= \int_{B_R}\frac{ |y|^2}{2R^2}dy+\frac{\pi R^2}{2}\log R-\frac{3}{8}\pi R^2+\pi R^2 \log c\nonumber\\
       &=\frac{\pi R^2}{2}\log R-\frac{\pi R^2}{8}+\pi R^2 \log c.
    \end{align}    
    Substituting \eqref{term1inLw1}, \eqref{term2inLw1}, and \eqref{term3inLw1} in \eqref{LW1is0} yields
    \begin{equation}\label{Lcw1at11}
        \Lom_{cw_1} (1)(x)= 2\pi R^2\log c.
    \end{equation}
    Consequently, for $c=1$ (i.e., $w=w_1$), 0 is an eigenvalue of \eqref{evproblem} and 1 is a corresponding eigenfunction. Since $\Delta Q_c = \frac{2}{R^2}$, by Proposition \ref{prop:Kernel}, any eigenfunction corresponding to 0 must be constant. However, for $c\neq 1$, we have $\Lom_{cw_1} (1)(x)\neq 0$, and hence $\Lom_{cw_1}(u)\neq 0$ for any non-zero constant function $u.$  Hence, 0 can not be an eigenvalue of \eqref{evproblem} for the weight function $w=cw_1$, with $c\neq 1$. 
\end{example}

Next, we compute the weighted transfinite diameter of the disk $B_R$ for the weight function $W = \frac{1}{w_1}.$
\begin{proposition}\label{prop_ball_computation}
    Let $\Omega = B_R$ and  $W(x)=e^{-\left(\frac{|x|^2}{2R^2}+\frac{1}{2}\log R-\frac{3}{8}\right)}$. Then, $$T_{diam}(B_R,W)=1.$$
\end{proposition}
\begin{proof}
Recall the alternative definition of the transfinite diameter given in Remark~\ref{alternate_definition} by
        \begin{equation}
            T_{diam}(B_R,W)=e^{- I_{B_R}(\mu,W)},
        \end{equation}
       where $\mu$ is the equilibrium measure that attains the weighted Robin constant. Let $\nu$ be the normalized Borel measure on $B_R$ given by $d\nu(x) = \frac{1}{|B_R|}dx$. From \eqref{term3inLw1}, we have
\begin{align}
    \mathcal{U}^\nu(x)
    := \frac{1}{|B_R|} \int_{B_R} \log\!\left(\frac{1}{|x-y|}\right)dy
    = \frac{1}{\pi R^2}\Lom(1)(x)
    = -\frac{|x|^2}{2R^2} - \log R + \frac{1}{2}.
\end{align}
Moreover,
\begin{equation}
    \mathcal{U}^\nu(x) - \log W(x) = \frac{1}{8} - \frac{1}{2}\log R,
\end{equation}
which is a constant. Since $\mathcal{U}^\nu(x) - \log W(x)$ is a constant for every $x$, by \cite[Theorem~3.3, p.~44]{ebsaff1997}, $\nu$ is the equilibrium measure that attains the Robin constant 
$V(\ol{B}_R, W) = I_{\ol{B}_R}(\nu, W)$. 
In addition,
\begin{align}
    I_{B_R}(\nu, W)
    &= \frac{1}{|B_R|^2} 
       \iint\limits_{B_R \; B_R} 
       \log\!\left(\frac{1}{|x-y|W(x)W(y)}\right)dx\,dy \nonumber\\
    &= \frac{1}{|B_R|^2} 
       \int_{B_R} \Lom_{w_1}(1)(x)\,dx
       = 0 
       \quad (\text{by Example~\ref{example_disk}}).
\end{align}
Therefore, $T_{diam}(B_R, W) = 1.$
\end{proof}
\begin{remark}\label{rem:example_explanation_theorem1.9}
Let $w_1$ be as in Example~\ref{example_disk}. Since $\Delta \log (c w_1)=\frac{2\pi}{|B_R|}$, the function $\log (c w_1)$ is strictly subharmonic. Moreover, for $0<c<1$, it follows from \eqref{Lcw1at11} that $\langle \mathcal{L}_{c w_1}1,1\rangle<0$, and therefore $\mathcal{L}_{c w_1}$ admits a negative eigenvalue. In addition, for $0<c<1$, $T_{\mathrm{diam}}\left(B_R,\frac{1}{c w_1}\right)=\frac{1}{c^2}>1$. Therefore, the assumption that $\log w$ is superharmonic is not a necessary condition for the conclusion in Theorem \ref{theorem_weighted_trans}.
\end{remark}

Next, we prove the representation formula for the eigenfunctions of \eqref{evproblem}.  For $w = g = 1$, such a formula was derived in \cite[p.~367]{troutman1967}. Here, we follow a similar approach. Using this representation, we derive the maximum principle that provides insight into the geometric properties of these eigenfunctions.

\noindent \textbf{Proof of Theorem \ref{efreptheorem}:} 
Let $(\tau,\phi)$ be an eigenpair of \eqref{evproblem}. By Theorem \ref{theorem_0_ev}, we have $\tau\neq 0$. By evaluating $\mathcal{L}_{w}\phi$  at $x+r\zeta$ and integrating both sides over $\partial B_1$, we obtain
    \begin{align}\label{eqn:A1}
        &\tau \int_{\partial B_1}\phi(x+r\zeta)g(x+r\zeta)dS(\zeta)=  \int_{\partial B_1}\int_\Om \log\left(\frac{w(x+r\zeta)w(y)}{|x+r\zeta-y|}\right)\phi(y)dy dS(\zeta).
    \end{align} 
 Now,    
     \begin{align}
    \log\left(\frac{w(x+r\zeta)w(y)}{|x+r\zeta-y|}\right)&=\log w\left(x+r\zeta\right)+\log w(y)- \log (|x+r\zeta-y|)\nonumber\\ &=\log w\left(x+r\zeta\right)+\log\left(\frac{w(y)}{|x-y|}\right)-\log\left(\frac{|x+r\zeta-y|}{|x-y|}\right)\nonumber\\ 
    &=\log w\left(x+r\zeta\right)+\log\left(\frac{w(y)}{|x-y|}\right)-\log\left(|x_0|^{-1}|x_0+\zeta| \right),\nonumber\\
    \end{align}
     where $x_0=\frac{x-y}{r}$. Since $\log w$ is harmonic,  we obtain
    \begin{align}\label{eqn:A2}
        \int_\Om \left[\int_{\partial B_1}  \log w(x+r\zeta)dS(\zeta)\right] \phi(y)dy=2\pi\int_\Om \log w(x)\phi(y)dy.
    \end{align}  
    Furthermore, we have
    \begin{align}\label{eqn:A4}
        \int_\Om \left[\int_{\partial B_1} \log\left(\frac{w(y)}{|x-y|}\right)dS(\zeta)\right] \phi(y)dy=2\pi\int_\Om \log\left(\frac{w(y)}{|x-y|}\right)\phi(y)dy,
    \end{align}  
    and 
   \begin{align}\label{eqn:A5}
        \int_{\partial B_1} \log\left(\frac{|x+r\zeta-y|}{|x-y|}\right)dS(\zeta)&=-2\pi\log(|x_0|)+\int_{\partial B_1}\log\left(|x_0+\zeta| \right)dS(\zeta)\\
        &=\begin{cases}
           -2\pi  \log(|x_0|), & |x_0|\leq 1,\\
           0, &|x_0|>1,
        \end{cases}
    \end{align} 
      where the last equalities are obtained from the following classical result of the potential theory (see, for example \cite[p.22]{ebsaff1997}) 
    \begin{equation*}
        \int_{\partial B_r(0)} \log (|x_0+r\zeta|)dS(\zeta) =\begin{cases}
            2\pi \log(r), & |x_0|\leq r,\\
            2\pi \log (|x_0|), &|x_0|>r.
        \end{cases}
    \end{equation*}
   Therefore, 
    \begin{equation}\label{eqn:A3}
        \int_\Om \left[\int_{\partial B_1} \log\left(\frac{|x+r\zeta-y|}{|x-y|}\right)dS(\zeta) \right]\phi(y)dy=-2\pi \int_{B_r(x)}\log\left( \frac{|x - y|}{r} \right)\phi(y)dy.
    \end{equation}
   Applying Fubini’s theorem to the right-hand side of \eqref{eqn:A1}, and substituting \eqref{eqn:A2}, \eqref{eqn:A4}, and \eqref{eqn:A3}, we get
   \begin{align}
        \tau\int_{\partial B_1} &\phi(x+r\zeta)g(x+r\zeta) dS(\zeta)
        = \int_\Om\int_{\partial B_1} \log\left(\frac{w(x+r\zeta)w(y)}{|x+r\zeta-y|}\right)dS(\zeta)\phi(y)dy\\        
        &=2\pi \Lom_w(\phi)(x)+2\pi \int_{B_r(x)}\log\left( \frac{|x - y|}{r} \right)\phi(y)dy\\
        &=2\pi \tau \phi(x)g(x) +2\pi \int_{B_r(x)}\log\left( \frac{|x - y|}{r} \right)\phi(y)dy
        \end{align} 
    Thus,
    \begin{equation}
        \tau \phi(x)g(x)=\frac{\tau}{2\pi} \int_{\partial B_1} \phi(x+r\zeta)g(x+r\zeta) dS(\zeta)-\int_{B_r(x)}\log\left( \frac{|x - y|}{r} \right)\phi(y)dy.
    \end{equation}
\qed
\begin{corollary}\label{cor_efrep}
Let $w\in\cpom$ be such that $\log(w)$ is a harmonic function on $\Om$, and let $(\tau,\phi)$ be an eigenpair of \eqref{evproblem}. Then:
\begin{enumerate}[(i)]
    \item If $\tau > 0$, then $\phi g$ cannot have a positive local minimum or a negative local maximum in $\Omega$.
    \item If $\tau < 0$, then $\phi g$ cannot have a negative local minimum or a positive local maximum in $\Omega$.
\end{enumerate}
    \end{corollary}
\begin{proof}
    Let $\tau > 0$. Suppose that $\phi g$ admits a positive local minimum at some point $x_0 \in \Omega$. Then there exists a neighborhood $B_r(x_0) \subset \Omega$ such that $\phi(x) g(x)\ge \phi(x_0)g(x_0)>0$, $\forall\, x\in B_r(x_0)$. Hence,
    \begin{equation}\label{phig1}
        \phi(x_0) g(x_0)\leq \frac{1}{2\pi} \int_{\partial B_1} \phi(x_0 + r y)g(x_0+ry)\, dS(y).
    \end{equation}
    On the other hand, by \eqref{eqn:repn}, we also have
    \begin{align*}
        \phi(x_0)g(x_0) &= \frac{1}{2\pi} \int_{\partial B_1} \phi(x_0 + r y) g(x_0+ry) \, dS(y) - \frac{1}{\tau}\int_{ B_r(x_0)} \log\left( \frac{|x_0 - \zeta|}{r} \right) \phi(\zeta) \, d\zeta\\
        &>\frac{1}{2\pi} \int_{\partial B_1} \phi(x_0 + r y)g(x_0+ry)\, dS(y).
    \end{align*}
    A contradiction. Thus, $\phi g$ does not admit a positive local minimum in $\Omega$. Considering the eigenfunction $-\phi$ in the above case, we conclude that $\phi g$ does not admit a negative local maximum in $\Omega$. The proof for $\tau<0$ follows a similar approach.
\end{proof}
\begin{remark}\label{rem:efrep}
Assume that $\log(w)$ is harmonic in $\Om$. For an eigenpair $(\tau,\phi)$ of \eqref{evproblem}, from the previous corollary, we can make the following observations:
\begin{enumerate}[$(i)$]   
    \item If $\phi g$ is constant in a neighbourhood in $\Om$, then this constant becomes a local extremum for $\phi g.$ Therefore, $\phi$ must be zero in this neighbourhood.
    \item If $\tau < 0$ and $\phi$ is sign-changing, then the global extrema of $\phi g$ on $\overline{\Omega}$ are attained on the boundary of $\Omega.$ 
    Thus $\phi$ can not be identically zero on $\partial \Omega.$
    \item \textbf{Maximum principle:} Let $\tau < 0$ and  $\phi\geq 0$ on $\partial\Omega.$ Then we must have $\phi\geq 0$ in $\Omega.$ If $\phi$ changes sign, then $\phi g$ admits a negative global minimum in $\Omega$, a contradiction.  
\end{enumerate}  
\end{remark}

\bibliographystyle{abbrvurl}
\bibliography{Reference}

\end{document}